\documentclass[11pt]{article}
\usepackage{amsmath}
\usepackage{amssymb}
\usepackage{cases}
\usepackage{amsfonts}
\usepackage{mathrsfs}
\usepackage{graphicx}
\usepackage{setspace}
  \usepackage{paralist}
  \usepackage{longtable}
   \usepackage{multirow}
    \usepackage{rotating}
\usepackage{fancyhdr}

\usepackage{amsmath, amsthm, amssymb}
\usepackage{authblk}
\usepackage{graphicx}
\usepackage{float}
\usepackage{hyperref}
\usepackage[margin=1in]{geometry}
\usepackage{comment}
\usepackage{color}
\allowdisplaybreaks[4]
\numberwithin{equation}{section}
\date{}

\newtheorem{theorem}{Theorem}[section]
\newtheorem{proposition}[theorem]{Proposition}
\newtheorem{lemma}[theorem]{Lemma}
\newtheorem{corollary}[theorem]{Corollary}
\newtheorem{remark}{Remark}[section]
\newtheorem{definition}[theorem]{Definition}

\newcommand{\be}{\begin{equation}}
\newcommand\ee{\end{equation}}
\newcommand\bes{\begin{eqnarray}}
\newcommand\ees{\end{eqnarray}}
\newcommand\bess{\begin{eqnarray*}}
\newcommand\eess{\end{eqnarray*}}

\title{Stability and instability for compressible Navier-Stokes equations with Yukawa potential}

\author{Juanzi Cai$^1$\ \ \ Zhigang Wu$^1$\thanks{zgwu@dhu.edu.cn}
\ \ \ Guochun Wu$^2$}

\begin{document}

\maketitle
\renewcommand{\thefootnote}{\fnsymbol{footnote}}

\footnotetext{1. Department of Mathematics, Donghua University, Shanghai 201620, P.R. China.}
\footnotetext{2. \!\!\!Fujian Province University Key Laboratory of Computational Science, School of Mathematical Sciences, Huaqiao University, Quanzhou 362021, P.R. China.}

\renewcommand{\thefootnote}{\arabic{footnote}}


\textbf{{\bf Abstract:}} In this paper, we first consider global well-posedness and long time behavior of compressible Navier-Stokes equations with Yukawa-type potential in $L^p$-framework under the stability condition $P'(\bar\rho)+\gamma\bar\rho>0$. Here $\bar\rho>0$ is the background density, $P$ is the pressure and $\gamma\in\mathbb{R}$ is Yukawa coefficient. This is a continuity work of Chikami \cite{chikami1} concerning on local existence and blow-up criterion.
On the other hand, we study the instability of the linear and nonlinear problem of the system when $P'(\bar\rho)+\gamma\bar\rho<0$ in the Hadamard sense.


{\textbf{AMS Subject Classification :} 35Q30, 74H40, 76N10.

	\section{Introduction}
	
\quad\quad In this paper, our concern is
the compressible Navier-Stokes equations with Yukawa-type potential \cite{chikami1}:
\begin{equation} \label{1.1}
\left\{\begin{array}{ll}
\rho_{t}+{\rm div}(\rho \mathbf{u})=0,\\[2mm]
(\rho \mathbf{u})_{t}+{\rm div}(\rho \mathbf{u}\otimes \mathbf{u})+\nabla P(\rho)+\gamma\rho\nabla\phi=\mu\Delta\mathbf{u}+(\mu+\mu')\nabla{\rm div}\mathbf{u},\\[2mm]
-\Delta\phi+\phi=\rho-\bar{\rho},\ \ \lim\limits_{|x|\rightarrow\infty}\rho(x,t)=\bar{\rho}>0,\ \ (x,t)\in\mathbb{R}^d\times\mathbb{R}_+,
\end{array}\right.
\end{equation}
where the unknown functions $\rho(x,t),\ \mathbf{u}(x,t)$ and $\phi(x,t)$ represent fluid density, velocity and the potential force exerted in the fluid, respectively. $P=P(\rho)$ is the pressure given by a smooth function depending on $\rho$, and $\mathbf{u}\otimes\mathbf{u}$ denotes the tensor product of two vectors. The viscosity constant $\mu,\mu'$ satisfy $\mu>0$ and $2\mu+\mu'>0$. The constant $\gamma\in\mathbb{R}$ may be arbitrary and is essential on its sign.

When $\gamma=0$, (\ref{1.1}) reduces to compressible Navier-Stokes system, which describes the motion of a barotropic viscous compressible flow. Due to the significance of the physical background, this model has been widely studied in past decades, see
\cite{feireisl1,feireisl4,hoff2,hoff3,jiang1,jiang2,kawashima,liu,wen1,wen3,xin1,xin2} and references therein. Matsumura and Nishida \cite{mastsumura1,mastsumura2} first obtained the global existence of small solution in $H^3(\mathbb{R}^3)$, and further got the $L^2$-decay rate of the solution as the heat equation under an additional assumption: the initial perturbation is small in $L^1$. Later on, Ponce \cite{ponce} gave the optimal $L^p$ $(p\geq2)$ decay rates of the solution and its first and second derivatives when the small perturbation is in $H^l\cap W^{l,1}$ with $l\geq4$, Duan, Liu, Ukai and Yang \cite{duan1} studied $L^p$ $(2\leq p\leq 6)$ decay rate for the system with external force and they need not the smallness of the initial perturbation in $L^1$-space. By replacing $L^1$-space by $\dot{B}_{1,\infty}^{-s}$ with $s\in[0,1]$, Li and Zhang \cite{li} achieve a faster $L^2$-decay rate of  the solution as $(1+t)^{-\frac{3}{4}+\frac{s}{2}}$. Guo and Wang \cite{guo1} developed a pure energy method to derive the $L^2$-decay rate of the solution and its derivative in $H^l\cap\dot{H}^{-s}$ with $l\geq3$ and $s\in[0,\frac{3}{2})$. In particular, they arrived at $\|\nabla^k(\rho-1,\mathbf{u})\|_{L^2}\lesssim(1+t)^{-\frac{k+s}{2}},\ 0\leq k\leq l-1$. Recently, Hu and Wu \cite{hu} established optimal rates of decay for solutions to the isentropic compressible Navier-Stokes equations with discontinuous initial data.

In the framework of critical spaces for the system (\ref{1.1}) with $\gamma=0$, a breakthrough was provided by Danchin \cite{danchin1},
where the author established the global existence of strong solutions to  for the initial data in the vicinity of the equilibrium in $(\dot{B}_{2,1}^{\frac{d}{2}}\cap\dot{B}_{2,1}^{\frac{d}{2}-1})\times\dot{B}_{2,1}^{\frac{d}{2}-1}$ with $d\geq2$. Subsequently, Charve and Danchin \cite{charve} and Chen, Miao and Zhang \cite{chen2} extended that result to the general $L^p$ critical Besov spaces. Haspot \cite{haspot} obtained the same results as in \cite{charve,chen2} by employing a more direct energy approach based on using Hoff's viscous effective
flux in \cite{hoff1}. Peng and Zhai \cite{peng} proved the global existence for $d$-dimensional compressible Navier-Stokes equations without heat conductivity for $d\geq2$ in $L^p$-framework. The critical Besov space in \cite{charve,chen2,haspot} for the compressible Navier-Stokes equations can be regraded as the largest one in which the system is well-posed, since Chen, Miao and Zhang verified the ill-posedness in \cite{chen3}.

When $\gamma\neq0$, (\ref{1.1}) becomes the compressible Navier-Stokes equations with a Yukawa-potential,
which is a simplified hydrodynamical model describing the nuclear matter \cite{chikami1,ducomet}. There are few results on this model. Chikami \cite{chikami1} studied local existence and blow-up criterion for the Cauchy problem of (\ref{1.1}) in critical $L^p$-space for any $\gamma\in\mathbb{R}$. Chen, Wu, Zhang and Zou \cite{chen4} obtained the global existence of small solution in $H^3$-space and the $L^2$-decay rate when the $L^1$-norm of the initial data is bounded when $P'(\bar\rho)>0$ and $\gamma>0$ with $\bar\rho>0$.

In this paper, our first goal is to study the global existence of the small strong solution to the Cauchy problem of (\ref{1.1}) in critical $L^p$-space based on the local existence in \cite{chikami1}. In particular, we achieve it under the stability condition $P'(\bar\rho)+\gamma\bar\rho>0$, which is naturally weaker than that in \cite{chen4}. And we also gave the $L^p$-decay rate of the strong solution. To this end, we first reformulate the system as follows. Based on the state equation, one has
\begin{equation*}
	\alpha_1\triangleq\frac{P'(\bar{\rho})}{\bar{\rho}},\ \  \alpha_2\triangleq\bar{\rho},\ \
	\alpha_3\triangleq\frac{\mu}{\bar{\rho}},\ \
	\alpha_4\triangleq\frac{\mu+\mu'}{\bar{\rho}}.
\end{equation*}
 By denoting
\begin{equation*}
	\tilde{\rho}\triangleq \rho-\bar{\rho},\ \ F(\rho)\triangleq\frac{P'(\rho+\bar{\rho})}{\rho+\bar{\rho}},\ \
	k(\tilde{\rho})\triangleq\frac{\rho}{\rho+\bar{\rho}},
\end{equation*}
then, the new system on the variables $\tilde{\rho},\mathbf{u}$ reads
\begin{equation}\label{1.4}
	\begin{cases}
		\partial_t\tilde{\rho}+\alpha_2{\rm div}\mathbf{u}=-\mathbf{u}\cdot\nabla \tilde{\rho}-\tilde{\rho}{\rm div}\mathbf{u},\\
		\partial_t\mathbf{u}-\alpha_3\Delta \mathbf{u}-\alpha_4\nabla{\rm div}\mathbf{u}+\alpha_1\nabla\tilde{\rho}+\gamma\nabla\phi\\
\ \ \ \ \ \ \ \ \ =-\mathbf{u}\cdot\nabla \mathbf{u}-F(\tilde{\rho})\nabla \tilde{\rho}-\alpha_3k(\tilde{\rho})\Delta \mathbf{u}-\alpha_4k(\tilde{\rho})\nabla{\rm div}\mathbf{u},\\
		-\Delta\phi+\phi=\tilde{\rho}.
	\end{cases}
\end{equation}
We consider the initial value problem of (\ref{1.4}) with the initial data
\begin{equation}\label{1.5}
	(\tilde{\rho}(x,t),\mathbf{u}(x,t))|_{t=0}=(\tilde{\rho}_0(x),\mathbf{u}_0(x))\rightarrow(0,0)\ \ {\rm for}\ |x|\rightarrow\infty.
\end{equation}
Solving the third equation for $\phi$ and plugging it into the second one, the equations for $(\tilde{\rho},\mathbf{u})$ are written as follows:
\begin{equation}\label{1.5(0)}
	\begin{cases}
		\partial_t\tilde{\rho}+\alpha_2{\rm div}{\mathbf{u}}=-\mathbf{u}\cdot\nabla \tilde{\rho}-\tilde{\rho}{\rm div}\mathbf{u},\\
		\partial_t\mathbf{u}-\alpha_3\Delta \mathbf{u}-\alpha_4\nabla{\rm div}\mathbf{u}+\alpha_1\nabla \tilde{\rho}+\gamma\frac{\nabla}{1-\Delta}\tilde{\rho}\\
\ \ \ \ \ \ \ \ \ \ \ \ \ \ \ =-\mathbf{u}\cdot\nabla \mathbf{u}-F(\tilde{\rho})\nabla \tilde{\rho}-\alpha_3k(\tilde{\rho})\Delta \mathbf{u}-\alpha_4k(\tilde{\rho})\nabla{\rm div}\mathbf{u}.
	\end{cases}
\end{equation}
 Note that when $\gamma=0$, the stability condition in this paper is consistent with the barotropic compressible Navier-Stokes equations  considered in \cite{charve,chen2,danchin1,haspot}. Additionally, this stability condition can be found in both spectral analysis and energy estimate. In fact, due to the presence of Bessel potential $(1-\Delta)^{-1}$ when one use $\rho$ to represent $\phi$ though the third equation in (\ref{1.1}), we shall reconstruct Green's function of (\ref{1.1}) to derive the estimate of low frequency in Besov space compared to the classical compressible Navier-Stokes equations. In particular, we divide the term $\gamma\frac{\nabla}{1-\Delta}\tilde{\rho}$ into $\gamma\nabla \tilde{\rho}$ and $\gamma\frac{\Delta}{1-\Delta}\nabla\tilde{\rho}$, then in the low frequency, the term $(\frac{P'(\bar\rho)}{\bar\rho}+\gamma)\nabla\rho$ is corresponding to $\frac{P'(\bar\rho)}{\bar\rho}$ in the classical compressible Navier-Stokes equations.  As a result, we need not $P'(\bar\rho)>0$ due to the presence of Yukawa potential.

The main results on the Cauchy problem (\ref{1.4})-(\ref{1.5}) are stated as follows. First of all, we denote homogeneous Besov space (its definition is in Appendix) $\dot{B}_{p,r}^\sigma:=\dot{B}_{p,r}^\sigma(\mathbb{R}^d)$ without confusion.
\begin{theorem}[Global existence]\label{Theorem 1.1}  Assume that $P'(\bar\rho)+\gamma\bar\rho>0$.
	Let $d\geq2$ and $p$ fulfills
	\begin{equation}\label{1.2.6}
		2\leq p\leq\min(4,2d/(d-2))\ and, \ additionally,\
		p\neq4\ if\ d=2.
	\end{equation}
	For $(\tilde{\rho}_0^l,\mathbf{u}_0^l)\in\dot{B}_{2,1}^{\frac{d}{2}-1}$ and $\tilde{\rho}_0^h\in\dot{B}_{p,1}^{\frac{d}{p}}$, $\mathbf{u}_0^h\in\dot{B}_{p,1}^{\frac{d}{p}-1}$, if there exists a small constant $c_0$ such that
	\begin{equation}\label{1.2.7}
		\|(\tilde{\rho}_0^l,\mathbf{u}_0^l)\|_{\dot{B}_{2,1}^{\frac{d}{2}-1}}+\|\tilde{\rho}_0^h\|_{\dot{B}^{\frac{d}{p}}_{p,1}}+\|\mathbf{u}_0^h\|_{\dot{B}_{p,1}^{\frac{d}{p}+1}}\leq c_0,
	\end{equation}
	then the Cauchy problem (\ref{1.4})-(\ref{1.5}) has a unique global solution $(\tilde{\rho},\mathbf{u})$ such that
	\begin{equation*}
		\begin{split}
			&\tilde{\rho}^l\in C_b(\mathbb{R}^+,\dot{B}_{2,1}^{\frac{d}{2}-1})\cap L^1(\mathbb{R}^+,\dot{B}_{2,1}^{\frac{d}{2}+1}),\ \tilde{\rho}^h\in C_b(\mathbb{R}^+,\dot{B}_{p,1}^{\frac{d}{p}})\cap L^1(\mathbb{R}^+,\dot{B}_{p,1}^{\frac{d}{p}}),\\
			&\mathbf{u}^l\in C_b(\mathbb{R}^+,\dot{B}_{2,1}^{\frac{d}{2}-1})\cap L^1(\mathbb{R}^+,\dot{B}_{2,1}^{\frac{d}{2}+1}),\ \mathbf{u}^h\in C_b(\mathbb{R}^+,\dot{B}_{p,1}^{\frac{d}{p}-1})\cap L^1(\mathbb{R}^+,\dot{B}_{p,1}^{\frac{d}{p}+1}).
		\end{split}
	\end{equation*}
	Moreover, there exists some positive constant $C$ independent of $t$ such that
	\begin{equation*}\label{1.7}
		\mathcal{X}(t)\leq Cc_0,
	\end{equation*}
	with
	\begin{equation*}
		\begin{split}
			\mathcal{X}(t)\triangleq&\|(\tilde{\rho}^l,\mathbf{u}^l)\|_{\tilde{L}_t^\infty(\dot{B}_{2,1}^{\frac{d}{2}-1})}+\|(\tilde{\rho}^l,\mathbf{u}^l)\|_{L_t^1(\dot{B}_{2,1}^{\frac{d}{2}+1})}\\
			&+\|\tilde{\rho}^h\|_{\tilde{L}_t^\infty(\dot{B}_{p,1}^{\frac{d}{p}})}+\|\mathbf{u}^h\|_{\tilde{L}_t^\infty(\dot{B}_{p,1}^{\frac{d}{p}-1})}
			+\|\tilde{\rho}^h\|_{L_t^1(\dot{B}_{p,1}^{\frac{d}{p}})}+\|\mathbf{u}^h\|_{L_t^1(\dot{B}_{p,1}^{\frac{d}{p}+1})}.
		\end{split}
	\end{equation*}
\end{theorem}


\begin{remark}\label{Remark 1.1} Under the stability condition $P'(\bar\rho)+\gamma\bar\rho>0$, we obtain the global existence of strong solution in $L^p$-framework, which is a complement of the local existence in \cite{chikami1}. On the other hand, we also refine the existence result in \cite{chen4}, where the authors gave the global existence and $L^2$-decay rate in Sobolev space under a stronger condition $P'(\bar\rho)>0$ and $\gamma>0$.
\end{remark}

The optimal time-decay estimates of solutions are obtained as follows.

\begin{theorem}\label{Theorem 1.2}  Assume that $P'(\bar\rho)+\gamma\bar\rho>0$, $d\geq2$ and $(\tilde{\rho},\mathbf{u})$ is the global solution of the Cauchy problem (\ref{1.4})-(\ref{1.5}) given in Theorem \ref{Theorem 1.1}. Furthermore, suppose that the real number $\sigma$ satisfies
	\begin{equation}\label{1.3.8}
		1-\frac{d}{2}<\sigma\leq\sigma_0\triangleq\frac{2d}{p}-\frac{d}{2}.
	\end{equation}
Then, when $\|(\tilde{\rho}_0^l,\mathbf{u}_0^l)\|_{\dot{B}_{2,\infty}^{-\sigma}}$ is bounded, it holds that
	\begin{equation}\label{1.3.9}
		\|\Lambda^{\sigma_1}(\tilde{\rho},\mathbf{u})\|_{L^p}\leq C(1+t)^{-\frac{d}{2}(\frac{1}{2}-\frac{1}{p})-\frac{\sigma+\sigma_1}{2}},\ if\ -\sigma-d(\frac{1}{2}-\frac{1}{p})<\sigma_1\leq\frac{d}{p}-1.
	\end{equation}
\end{theorem}

\begin{remark}In fact, by choosing $p=2,d=3$ and $\sigma=\frac{3}{2}$ in Theorem \ref{Theorem 1.2}, one has $\|(\tilde{\rho},\mathbf{u})\|_{L^2}\lesssim(1+t)^{-\frac{3}{4}}$, which were shown by Chen $et.\ al.$ \cite{chen4} in Sobolev space with higher regularity.
\end{remark}

On the contrary, when $P'(\bar\rho)+\gamma\bar\rho<0$, we shall study the linear and nonlinear instability of the Cauchy problem (\ref{1.4})-(\ref{1.5}) in Sobolev space based on the idea in \cite{guo2,jang,jiang3,jiang4,wangy} by constructing the initial data relying on the maximum of the real part of the eigenvalue. For simplicity, we only state the results for three dimensional case, and we emphasize that two dimensional case can be obtained similarly.

\begin{theorem}\label{l 1.3}{\rm[Linear\ instability]} Assume that $P'(\bar\rho)+\gamma\bar\rho<0$. Then there exists one positive constant $\Theta$ and for any $\bar\Theta\in(0,\frac{\Theta}{2}]$, the linear system of (\ref{1.1}) admits unstable solution $(\rho_{\bar\Theta}^l,\mathbf{u}_{\bar\Theta}^l)$ with the initial data $(\rho_{0,\bar\Theta}^l,\mathbf{u}_{0,\bar\Theta}^l)$ satisfying
\begin{equation}\label{1.18}
(\rho_{\bar\Theta}^l-\bar\rho,\mathbf{u}_{\bar\Theta}^l)\in C^0(0,\infty; H^3(\mathbb{R}^3)),
\end{equation}
and
\begin{equation}\label{1.19}
\begin{split}
	e^{(\Theta-\bar\Theta)t}\|\rho_{0,\bar\Theta}^l\|_{L^2}\leq\|\rho_{\bar\Theta}^l\|\leq e^{\Theta t}\|\rho_{0,\bar\Theta}^l\|_{L^2},\\
e^{(\Theta-\bar\Theta)t}\|\mathbf{u}_{0,\bar\Theta}^l\|_{L^2}\leq\|\mathbf{u}_{\bar\Theta}^l\|\leq e^{\Theta t}\|\mathbf{u}_{0,\bar\Theta}^l\|_{L^2}.
\end{split}
\end{equation}
Here the initial data $(\rho_{0,\bar\Theta}^l,\mathbf{u}_{0,\bar\Theta}^l)$ depends on $\bar\Theta$ and
satisfy $\|\rho_{0,\bar\Theta}^l-\bar\rho\|_{L^2}\|\mathbf{u}_{0,\bar\Theta}^l\|_{L^2}>0$.
\end{theorem}

\begin{remark}The unstable solution for the linear problem  is constructed in (\ref{4.9}) by using cut-off technique as in \cite{chen5}, where they consider a quasi-linear hyperbolic-parabolic model for vasculogenesis. In fact, the main difference in the proof of instability for these two models is arising from the difference between the damped mechanism in \cite{chen5} and the diffusion mechanism for the system (\ref{1.1}). See the details in Section 4.
\end{remark}

\begin{theorem}\label{l 1.4}{\rm[Nonlinear\ instability]} Assume that $P'(\bar\rho)+\gamma\bar\rho<0$. Then the steady state $(\bar\rho,0)$ of the system (\ref{1.1}) is unstable in the Hadamard sense. In other words, there exist constants $\epsilon_0>0$ and $\delta_0>0$ such that for any $\delta\in(0,\delta_0)$, and the initial data $(\rho_0,\mathbf{u}_0)=\delta(\rho_{0,\bar\Theta}^l,\mathbf{u}_{0,\bar\Theta}^l)$ with $(\rho_{0,\bar\Theta}^l,\mathbf{u}_{0,\bar\Theta}^l)$ defined in Theorem \ref{l 1.3}, there exists a unique strong solution $(\rho,\mathbf{u})$ of the nonlinear problem (\ref{1.1}), such that
\begin{equation}\label{1.16}
\|\rho(T^\delta)-\bar\rho\|_{L^2}\geq\epsilon_0,\ \ \ \|\mathbf{u}(T^\delta)\|_{L^2}\geq\epsilon_0,
\end{equation}
for some escape time $T^\delta=\frac{1}{\Theta}\ln\frac{2\epsilon_0}{\delta}\in(0,T^{max})$. Here $T^{max}$ denotes the maximal time of existence of the solution $(\rho,\mathbf{u})$.
\end{theorem}

\begin{remark}The critical case $P'(\bar\rho)+\gamma\bar\rho=0$ is naturally interesting, and its stability or instability will be considered in future.
\end{remark}

\noindent{\rm\bf Notations: }
The letter $C$ stands for a generic positive constant whose meaning is clear from the context.
We write $f\lesssim g$ instead of $f\leq Cg$. For operators $A$ and $B$, we denote the commutator $[A,B]=AB-BA$.

The remainder of the paper is organized as follows.  In Section 2, we mainly give the a $priori$ estimate of the solution in $L^p$-framework, which together with local existence gives the global existence. In Section 3, we derive the $L^p$-decay rate of the solution given in Section 2 under the initial assumption in some negative Besov space. We study the instability of the system in Section 4. Finally, we recall the Littlewood-Paley decomposition, Besov spaces and related analysis tools in Appendix.

\section{The Proof of Global Existence}

\quad\quad In this section, we will divide into four subsections to verify the global existence in Theorem \ref{Theorem 1.1}. In the first two subsections, we show the estimate in the low frequency. In the last two subsections, we present the estimate in the high frequency. Then, we eventually acquire the global existence in Theorem \ref{Theorem 1.1}.

Recall the system \eqref{1.5(0)}:
\begin{equation}\label{2.1}
	\begin{cases}
		\partial_t\tilde{\rho}+\alpha_2{\rm div}{\mathbf{u}}=N_1,\\
		\partial_t\mathbf{u}-\alpha_3\Delta \mathbf{u}-\alpha_4\nabla{\rm div}\mathbf{u}+\alpha_1\nabla \tilde{\rho}+\gamma\frac{\nabla}{1-\Delta}\tilde{\rho}=N_2,\\
	\end{cases}
\end{equation}
where
\begin{equation}\label{2.2}
	\begin{split}
		N_1:=&-\mathbf{u}\cdot\nabla \tilde{\rho}-\tilde{\rho}{\rm div}\mathbf{u},\\
		N_2:=&-\mathbf{u}\cdot\nabla \mathbf{u}-F(\tilde{\rho})\nabla \tilde{\rho}-\alpha_3k(\tilde{\rho})\Delta \mathbf{u}-\alpha_4k(\tilde{\rho})\nabla{\rm div}\mathbf{u}.\\
	\end{split}
\end{equation}
Besides, for convenience, we define
\begin{equation*}
	\begin{split}
		\mathscr{E}_\infty(t):=&\|(\tilde{\rho},\mathbf{u})\|^l_{\dot{B}^{\frac{d}{2}-1}_{2,1}}+\|\tilde{\rho}\|^h_{\dot{B}^{\frac{d}{p}}_{p,1}}+\|\mathbf{u}\|^h_{\dot{B}^{\frac{d}{p}-1}_{p,1}},\\
		\mathscr{E}_1(t):=&\|(\tilde{\rho},\mathbf{u})\|^l_{\dot{B}^{\frac{d}{2}+1}_{2,1}}+\|\tilde{\rho}\|^h_{\dot{B}^{\frac{d}{p}}_{p,1}}+\|\mathbf{u}\|^h_{\dot{B}^{\frac{d}{p}+1}_{p,1}}.		
	\end{split}
\end{equation*}

\subsection{The estimate of $\mathbb{P}\mathbf{u}$ in the low frequency}
\quad\quad Firstly, we estimate $\mathbb{P}\mathbf{u}$ in the low frequency. We handle the second equation in \eqref{2.1} by taking the operator  $\mathbb{P}$ and get
\begin{equation*}
	\partial_t\mathbb{P}\mathbf{u}-\alpha_3\Delta\mathbb{P}\mathbf{u}=\mathbb{P}N_2.
\end{equation*}
Applying $\dot{\Delta}_j$ to the above equation, and taking the $L^2$ inner product with $\dot{\Delta}_j\mathbb{P}\mathbf{u}$ gives
\begin{equation}\label{2.3}
	\frac{1}{2}\frac{d}{dt}\|\dot{\Delta}_j\mathbb{P}\mathbf{u}\|^2_{L^2}+c\alpha_32^{2j}\|\dot{\Delta}_j\mathbb{P}\mathbf{u}\|^2_{L^2}\lesssim\|\dot{\Delta}_j\mathbb{P}N_2\|_{L^2}\|\dot{\Delta}_j\mathbb{P}\mathbf{u}\|_{L^2},
\end{equation}
where the term $c2^{2j}\|\dot{\Delta}_j\mathbb{P}\mathbf{u}\|^2_{L^2}$ is produced by Bernstein's inequality: there exists a positive constant $c$ so that
\begin{equation*}
	-\int_{\mathbb{R}^d}\Delta\dot{\Delta}_j\mathbb{P}\mathbf{u}\cdot \dot{\Delta}_j\mathbb{P}\mathbf{u}dx\ge c2^{2j}\|\dot{\Delta}_j\mathbb{P}\mathbf{u}\|^2_{L^2}.
\end{equation*}
After that, multiplying by $1/\|\dot{\Delta}_j\mathbb{P}\mathbf{u}\|_{L^2}2^{j(\frac{d}{2}-1)}$ on both hand side of \eqref{2.3} and integrating the inequality from 0 to $t$, we can acquire by summing up about $j\le j_0$ that
\begin{equation}\label{2.4}
	\|\mathbb{P}\mathbf{u}\|^l_{\widetilde{L}^\infty_{t}(\dot{B}^{\frac{d}{2}-1}_{2,1})}+\alpha_3\|\mathbb{P}\mathbf{u}\|^l_{L^1_{t}(\dot{B}^{\frac{d}{2}+1}_{2,1})}
	\lesssim
	\|\mathbb{P}\mathbf{u}_0\|^l_{\dot{B}^{\frac{d}{2}-1}_{2,1}}+\|N_2\|^l_{L^1_{t}(\dot{B}^{\frac{d}{2}-1}_{2,1})}.
\end{equation}
Now, we start to handle the terms in $N_2$ one by one. For the expression of $N_2$, we have
\begin{equation}\label{2.5}
	\begin{split}
		\|{N_2}\|^l_{\dot{B}^{\frac{d}{2}-1}_{2,1}}\lesssim&\|\mathbf{u}\cdot\nabla \mathbf{u}\|^l_{\dot{B}^{\frac{d}{2}-1}_{2,1}}+\|F(\tilde{\rho})\nabla \tilde{\rho}\|^l_{\dot{B}^{\frac{d}{2}-1}_{2,1}}\\
		&+\|k(\tilde{\rho})\Delta \mathbf{u}\|^l_{\dot{B}^{\frac{d}{2}-1}_{2,1}}+\|k(\tilde{\rho})\nabla{\rm div}\mathbf{u}\|^l_{\dot{B}^{\frac{d}{2}-1}_{2,1}}.
	\end{split}
\end{equation}
For the term $\|\mathbf{u}\cdot\nabla \mathbf{u}\|^l_{\dot{B}^{\frac{d}{2}-1}_{2,1}}$, using Bony decomposition, we have
\begin{equation*}
	\mathbf{u}\cdot \nabla \mathbf{u}=T_\mathbf{u}\nabla \mathbf{u} +R(\mathbf{u},\nabla \mathbf{u})+T_{\nabla \mathbf{u}}\mathbf{u}.
\end{equation*}
According to embedding theorem of Proposition \ref{A.5}, interpolation theorem of Proposition \ref{interpolation} and Proposition \ref{A.11}, we can get
\begin{equation}\label{2.6}
	\begin{split}
		\|T_\mathbf{u}\nabla \mathbf{u}+R(\mathbf{u},\nabla \mathbf{u})\|^l_{\dot{B}^{\frac{d}{2}-1}_{2,1}}
		\lesssim&\|\mathbf{u}\|_{\dot{B}^{\frac{d}{p}-1}_{p,1}}\|\nabla \mathbf{u}\|_{\dot{B}^{\frac{d}{p}}_{p,1}}
		\lesssim\|\mathbf{u}\|_{\dot{B}^{\frac{d}{p}-1}_{p,1}}\|\mathbf{u}\|_{\dot{B}^{\frac{d}{p}+1}_{p,1}}\\
		\lesssim&(\|\mathbf{u}\|^l_{\dot{B}^{\frac{d}{p}-1}_{p,1}}+\|\mathbf{u}\|^h_{\dot{B}^{\frac{d}{p}-1}_{p,1}})(\|\mathbf{u}\|^l_{\dot{B}^{\frac{d}{p}+1}_{p,1}}+\|\mathbf{u}\|^h_{\dot{B}^{\frac{d}{p}+1}_{p,1}})\\
		\lesssim&(\|\mathbf{u}\|^l_{\dot{B}^{\frac{d}{2}-1}_{2,1}}\!+\!\|\mathbf{u}\|^h_{\dot{B}^{\frac{d}{p}-1}_{p,1}})(\|\mathbf{u}\|^l_{\dot{B}^{\frac{d}{2}+1}_{2,1}}\!+\!\|\mathbf{u}\|^h_{\dot{B}^{\frac{d}{p}+1}_{p,1}})
		\lesssim\mathscr{E}_\infty(t)\mathscr{E}_1(t),
	\end{split}
\end{equation}
and
\begin{equation}\label{2.7}
	\begin{split}
		\|T_{\nabla \mathbf{u}}\mathbf{u}\|^l_{\dot{B}^{\frac{d}{2}-1}_{2,1}}
		\lesssim&\|\nabla \mathbf{u}\|_{\dot{B}^{\frac{d}{p}-1}_{p,1}}\|\mathbf{u}\|_{\dot{B}^{\frac{d}{p}}_{p,1}}
		\lesssim\|\mathbf{u}\|^2_{\dot{B}^{\frac{d}{p}}_{p,1}}
		\lesssim\|\mathbf{u}\|_{\dot{B}^{\frac{d}{p}-1}_{p,1}}\|\mathbf{u}\|_{\dot{B}^{\frac{d}{p}+1}_{p,1}}\\
		\lesssim&(\|\mathbf{u}\|^l_{\dot{B}^{\frac{d}{2}-1}_{2,1}}+\|\mathbf{u}\|^h_{\dot{B}^{\frac{d}{p}-1}_{p,1}})(\|\mathbf{u}\|^l_{\dot{B}^{\frac{d}{2}+1}_{2,1}}+\|\mathbf{u}\|^h_{\dot{B}^{\frac{d}{p}+1}_{p,1}})
		\lesssim\mathscr{E}_\infty(t)\mathscr{E}_1(t).
	\end{split}
\end{equation}
Combining \eqref{2.6} and \eqref{2.7} gives
\begin{equation}\label{2.8}
	\|\mathbf{u}\cdot \nabla \mathbf{u}\|^l_{\dot{B}^{\frac{d}{2}-1}_{2,1}}\lesssim\mathscr{E}_\infty(t)\mathscr{E}_1(t).
\end{equation}
For the term $\|F(\tilde{\rho})\nabla\tilde{\rho}\|^l_{\dot{B}^{\frac{d}{2}-1}_{2,1}}$, we have
\begin{equation}
	\|F(\tilde{\rho})\nabla\tilde{\rho}\|^l_{\dot{B}^{\frac{d}{2}-1}_{2,1}}\lesssim\|F(\tilde{\rho})\nabla\tilde{\rho}^l\|^l_{\dot{B}^{\frac{d}{2}-1}_{2,1}}+\|F(\tilde{\rho})\nabla\tilde{\rho}^h\|^l_{\dot{B}^{\frac{d}{2}-1}_{2,1}}.
\end{equation}
First, to deal with the first term of the right-hand side, we use Proposition \ref{A.7} and Proposition \ref{A.13} and acquire that
\begin{equation}
	\begin{split}
		\|F(\tilde{\rho})\nabla\tilde{\rho}^l\|^l_{\dot{B}^{\frac{d}{2}-1}_{2,1}}\lesssim&\|F(\tilde{\rho})\|_{\dot{B}^{\frac{d}{p}}_{p,1}}\|\nabla\tilde{\rho}^l\|_{\dot{B}^{\frac{d}{2}-1}_{2,1}}
		\lesssim\|\tilde{\rho}\|_{\dot{B}^{\frac{d}{p}}_{p,1}}\|\tilde{\rho}^l\|_{\dot{B}^{\frac{d}{2}}_{2,1}}\\
		\lesssim&\|\tilde{\rho}^l\|_{\dot{B}^{\frac{d}{2}}_{2,1}}\|\tilde{\rho}^l\|_{\dot{B}^{\frac{d}{2}}_{2,1}}+\|\tilde{\rho}^h\|_{\dot{B}^{\frac{d}{p}}_{p,1}}\|\tilde{\rho}^l\|_{\dot{B}^{\frac{d}{2}}_{2,1}}\\
		\lesssim&\|\tilde{\rho}^l\|_{\dot{B}^{\frac{d}{2}-1}_{2,1}}\|\tilde{\rho}^l\|_{\dot{B}^{\frac{d}{2}+1}_{2,1}}+\|\tilde{\rho}^h\|^{\frac{1}{2}}_{\dot{B}^{\frac{d}{p}}_{p,1}}\|\tilde{\rho}^h\|^{\frac{1}{2}}_{\dot{B}^{\frac{d}{p}}_{p,1}}\|\tilde{\rho}^l\|^{\frac{1}{2}}_{\dot{B}^{\frac{d}{2}-1}_{2,1}}\|\tilde{\rho}^l\|^{\frac{1}{2}}_{\dot{B}^{\frac{d}{2}+1}_{2,1}}\\
		\lesssim&\mathscr{E}_\infty(t)\mathscr{E}_1(t)+(\mathscr{E}_\infty(t))^{\frac{1}{2}}(\mathscr{E}_1(t))^{\frac{1}{2}}(\mathscr{E}_\infty(t))^{\frac{1}{2}}(\mathscr{E}_1(t))^{\frac{1}{2}}\lesssim\mathscr{E}_\infty(t)\mathscr{E}_1(t).
	\end{split}
\end{equation}
For the second term, we use Bony decomposition again and then it follows that
\begin{equation}
	F(\tilde{\rho})\nabla\tilde{\rho}^h=T_{F(\tilde{\rho})}\nabla\tilde{\rho}^h +R(F(\tilde{\rho}),\nabla\tilde{\rho}^h)+T_{\nabla\tilde{\rho}^h}F(\tilde{\rho}),
\end{equation}
\begin{equation}
	\begin{split}
		\|T_{F(\tilde{\rho})}\nabla\tilde{\rho}^h\|^l_{\dot{B}^{\frac{d}{2}-1}_{2,1}}\lesssim&\|T_{F(\tilde{\rho})}\nabla\tilde{\rho}^h\|^l_{\dot{B}^{\frac{d}{2}-2}_{2,1}}		\lesssim\|F(\tilde{\rho})\|_{\dot{B}^{\frac{d}{p}-1}_{p,1}}\|\nabla\tilde{\rho}^h\|_{\dot{B}^{\frac{d}{p}-1}_{p,1}}\\
		\lesssim&(1+\|\tilde{\rho}\|_{\dot{B}^{\frac{d}{p}}_{p,1}})\|\tilde{\rho}\|_{\dot{B}^{\frac{d}{p}-1}_{p,1}}\|\tilde{\rho}^h\|_{\dot{B}^{\frac{d}{p}}_{p,1}}
		\lesssim(1+\|\tilde{\rho}\|_{\dot{B}^{\frac{d}{p}}_{p,1}})\|\tilde{\rho}\|_{\dot{B}^{\frac{d}{p}-1}_{p,1}}\|\tilde{\rho}^h\|_{\dot{B}^{\frac{d}{p}}_{p,1}}\\
		\lesssim&(1+\mathscr{E}_\infty(t))\mathscr{E}_\infty(t)\mathscr{E}_1(t)
		\lesssim\mathscr{E}_\infty(t)\mathscr{E}_1(t),\ \ \ (\mathscr{E}_\infty(t)<<1),
	\end{split}
\end{equation}
and
\begin{equation}
	\begin{split}
		&\|R(F(\tilde{\rho}),\nabla\tilde{\rho}^h)+T_{\nabla\tilde{\rho}^h}F(\tilde{\rho})\|^l_{\dot{B}^{\frac{d}{2}-1}_{2,1}}\\
\lesssim&\|\nabla\tilde{\rho}^h\|_{\dot{B}^{\frac{d}{p}-1}_{p,1}}\|F(\tilde{\rho})\|_{\dot{B}^{\frac{d}{p}}_{p,1}}
		\lesssim\|\tilde{\rho}^h\|_{\dot{B}^{\frac{d}{p}}_{p,1}}\|\tilde{\rho}\|_{\dot{B}^{\frac{d}{p}}_{p,1}}
		\lesssim\mathscr{E}_\infty(t)\mathscr{E}_1(t).
	\end{split}
\end{equation}
Hence, the term $\|F(\tilde{\rho})\cdot\nabla\tilde{\rho}\|^l_{\dot{B}^{\frac{d}{2}-1}_{2,1}}$ can be bounded as
\begin{equation}\label{2.23}
	\|F(\tilde{\rho})\cdot\nabla\tilde{\rho}\|^l_{\dot{B}^{\frac{d}{2}-1}_{2,1}}\lesssim\mathscr{E}_\infty(t)\mathscr{E}_1(t).
\end{equation}
As for  the last two terms $\|k(\tilde{\rho})\Delta \mathbf{u}\|^l_{\dot{B}^{\frac{d}{2}-1}_{2,1}}$ and $\|k(\tilde{\rho})\nabla{\rm div}\mathbf{u}\|^l_{\dot{B}^{\frac{d}{2}-1}_{2,1}}$, it is similar to use Bony decomposition and get
\begin{equation}\label{2.24}
	\|k(\tilde{\rho})\Delta \mathbf{u}\|^l_{\dot{B}^{\frac{d}{2}-1}_{2,1}}+\|k(\tilde{\rho})\nabla{\rm div}\mathbf{u}\|^l_{\dot{B}^{\frac{d}{2}-1}_{2,1}}\lesssim\mathscr{E}_\infty(t)\mathscr{E}_1(t).
\end{equation}
Plugging \eqref{2.8}, \eqref{2.23} and \eqref{2.24} into \eqref{2.5}, then $\|N_2\|^l_{\dot{B}^{\frac{d}{2}-1}_{2,1}}$ can be estimated as
\begin{equation}\label{2.25}
	\|N_2\|^l_{\dot{B}^{\frac{d}{2}-1}_{2,1}}\lesssim\mathscr{E}_\infty(t)\mathscr{E}_1(t).
\end{equation}
Consequently, inserting \eqref{2.25} into \eqref{2.4}, we acquire the low frequency $\mathbb{P}\mathbf{u}^l$ as
\begin{equation}\label{2.26}
	\|\mathbb{P}\mathbf{u}\|^l_{\widetilde{L}^\infty_{t}(\dot{B}^{\frac{d}{2}-1}_{2,1})}+\alpha_3\|\mathbb{P}\mathbf{u}\|^l_{L^1_{t}(\dot{B}^{\frac{d}{2}+1}_{2,1})}\\
	\lesssim \|\mathbb{P}\mathbf{u}_0\|^l_{\dot{B}^{\frac{d}{2}-1}_{2,1}}+\int_{0}^{t}\mathscr{E}_\infty(\tau)\mathscr{E}_1(\tau)d\tau.
\end{equation}

\subsection{The estimate of $(\tilde{\rho},\mathbf{u})$ in the low frequency}
\quad\quad To derive the estimate of $(\tilde{\rho},\mathbf{u})$ in the low frequency, we shall use the Green function method. Thus, for the system \eqref{2.1}:
\begin{equation}
	\begin{cases}
		\partial_t\tilde{\rho}+\bar{\rho}{\rm div}{\mathbf{u}}=N_1,\\
		\partial_t\mathbf{u}-\frac{\mu}{\bar{\rho}}\Delta \mathbf{u}-\frac{\mu+\mu'}{\bar{\rho}}\nabla{\rm div}\mathbf{u}+\frac{P'(\bar{\rho})}{\bar{\rho}}\nabla \tilde{\rho}+\gamma\frac{\nabla}{1-\Delta}\tilde{\rho}=N_2,\\
	\end{cases}
\end{equation}
by dividing the term $\frac{\nabla}{1-\Delta}\tilde{\rho}$ into $\nabla \tilde{\rho}$ and $\frac{\Delta}{1-\Delta}\nabla\tilde{\rho}$ to get
\begin{equation}\label{2.19}
	\begin{cases}
		\partial_t\tilde{\rho}+\bar{\rho}{\rm div}{\mathbf{u}}=0,\\
		\partial_t\mathbf{u}-\frac{\mu}{\bar{\rho}}\Delta \mathbf{u}-\frac{\mu+\mu'}{\bar{\rho}}\nabla{\rm div}\mathbf{u}+(\frac{P'(\bar{\rho})}{\bar{\rho}}+\gamma)\nabla \tilde{\rho}+\gamma\frac{\Delta}{1-\Delta}\nabla\tilde{\rho}=0.
	\end{cases}
\end{equation}
Applying operators $\mathcal{P}$ and $ \mathcal{Q} $ to the second equation and setting $\eta:=\frac{2\mu+\mu'}{\bar{\rho}}$, it follows that
\begin{equation}
	\begin{cases}
		\partial_t\tilde{\rho}+\bar{\rho}{\rm div}\mathcal{Q}\mathbf{u}=0,\\
		\partial_t\mathbf{u}-\eta\Delta \mathcal{Q}\mathbf{u}+(\frac{P'(\bar{\rho})}{\bar{\rho}}+\gamma)\mathcal{Q}\nabla \tilde{\rho}+\gamma\frac{\Delta}{1-\Delta}\mathcal{Q}\nabla\tilde{\rho}=0,\\
		\partial_t\mathcal{P}\mathbf{u}-\eta\Delta \mathcal{P}\mathbf{u}=0.\\
	\end{cases}
\end{equation}
According to homogeneous Besov spaces setting, it is equivalent to bound $\mathcal{Q}\mathbf{u}$ or $v=\Lambda^{-1}{\rm div}\mathbf{u}$. So we can consider
\begin{equation}
	\begin{cases}
		\partial_t\tilde{\rho}+\bar{\rho}\Lambda v=0,\\
		\partial_tv-\eta\Delta v-(\frac{P'(\bar{\rho})}{\bar{\rho}}+\gamma)\Lambda \tilde{\rho}-\gamma\frac{\Delta}{1-\Delta}\Lambda\tilde{\rho}=0,\\
		\partial_t\mathcal{P}\mathbf{u}-\eta\Delta \mathcal{P}\mathbf{u}=0.\\
	\end{cases}
\end{equation}
We just focus on the second and third equation of (2.21), namely
\begin{equation}
	\begin{cases}
		\partial_t\tilde{\rho}+\bar{\rho}\Lambda v=0,\\
		\partial_tv-\eta\Delta v-(\frac{P'(\bar{\rho})}{\bar{\rho}}+\gamma)\Lambda \tilde{\rho}-\gamma\frac{\Delta}{1-\Delta}\Lambda\tilde{\rho}=0.\\
	\end{cases}
\end{equation}
Now, taking the Fourier transform with respect to the space variable yields
\begin{equation}
	\frac{d}{dt}
	\left(
	\begin{array}{c}
		\widehat{\tilde{\rho}}\\
		\widehat{v}\\
	\end{array}\right)=A(\xi)
	\left(
	\begin{array}{c}
		\widehat{\tilde{\rho}}\\
		\widehat{v}\\
	\end{array}
	\right)
	\ \ {\rm with}\ \ A(\xi):=
	\left(
	\begin{array}{cc}
		0&-\bar{\rho}|\xi|\\
		(\frac{P'(\bar{\rho})}{\bar{\rho}}+\gamma)|\xi|-\gamma\frac{|\xi|^2}{1+|\xi|^2}|\xi|&-\eta|\xi|^2\\
	\end{array}
	\right).
\end{equation}
Based on a serious of calculations, we get the characteristic polynomial of $A(\xi)$ is
\begin{equation}\label{2.24}
	\lambda^2+\eta|\xi|^2\lambda+\bar{\rho}|\xi|^2\big(\frac{P'(\bar{\rho})}{\bar{\rho}}+\gamma-\gamma\frac{|\xi|^2}{1+|\xi|^2}\big)=0,
\end{equation}
and the discriminant of which is
\begin{equation}
	\delta(\xi)=|\xi|^2\left[\eta^2|\xi|^2-4\bar{\rho}\Big(\frac{P'(\bar{\rho})}{\bar{\rho}}+\gamma-\gamma\frac{|\xi|^2}{1+|\xi|^2}\Big)\right].
\end{equation}


For the low frequency regime, we need the stability condition  $P'(\bar{\rho})+\gamma\bar{\rho}>0$ and then consider two cases: $|\xi|\le\sqrt{\frac{2(P'(\bar{\rho})+\gamma\bar{\rho})}{\eta^2}}$ when $\gamma\leq0$, or $|\xi|\le \min \Big(\sqrt{\frac{2(P'(\bar{\rho})+\gamma\bar{\rho})}{\eta^2}},\sqrt{\frac{P'(\bar{\rho})+\gamma\bar{\rho}}{2\gamma\bar{\rho}}}\Big)$ when $\gamma>0$. Then, for both of two cases, there are two distinct complex conjugated eigenvalues to be presented, so we finally have
\begin{equation}
	\lambda_{\pm}(\xi)=-\frac{\eta|\xi|^2}{2}(1\pm iS(\xi)),
\end{equation}
where
\begin{equation}
	S(\xi)=\sqrt{\frac{4\bar{\rho}|\xi|^2(\frac{P'(\bar{\rho})}{\bar{\rho}}+\gamma-\gamma\frac{|\xi|^2}{1+|\xi|^2})-\eta^2|\xi|^4}{\eta^2|\xi|^4}}
	=\sqrt{\frac{4\bar{\rho}(\frac{P'(\bar{\rho})}{\bar{\rho}}+\gamma-\gamma\frac{|\xi|^2}{1+|\xi|^2})}{\eta^2|\xi|^2}-1}.
\end{equation}
Through a serious of computations, we can acquire
\begin{equation}
	\begin{split}
		\widehat{\tilde{\rho}}(t,\xi)=&e^{t\lambda_-(\xi)}\left(\frac{1}{2}(1+\frac{i}{S(\xi)})\widehat{\tilde{\rho}}_0(\xi)-\frac{i}{\eta|\xi|S(\xi)}\widehat{v}_0(\xi)\right)\\
		&+e^{t\lambda_+(\xi)}\left(\frac{1}{2}(1-\frac{i}{S(\xi)})\widehat{\tilde{\rho}}_0(\xi)+\frac{i}{\eta|\xi|S(\xi)}\widehat{v}_0(\xi)\right),
	\end{split}
\end{equation}
\begin{equation}
	\begin{split}
		\widehat{v}(t,\xi)=&e^{t\lambda_-(\xi)}\left(\frac{i}{\eta|\xi|S(\xi)}\widehat{\tilde{\rho}}_0(\xi)+\frac{1}{2}(1-\frac{i}{S(\xi)})\widehat{v}_0(\xi)\right)\\
		&+e^{t\lambda_+(\xi)}\left(-\frac{i}{\eta|\xi|S(\xi)}\widehat{\tilde{\rho}}_0(\xi)+\frac{1}{2}(1+\frac{i}{S(\xi)})\widehat{v}_0(\xi)\right).
	\end{split}
\end{equation}
Let $\xi\to0$, we write
\begin{equation}
	\begin{split}
		&\widehat{\tilde{\rho}}(t,\xi)\sim\frac{1}{2}e^{t\lambda_-(\xi)}\left(\widehat{\tilde{\rho}}_0(\xi)-i\widehat{v}_0(\xi)\right)+\frac{1}{2}e^{t\lambda_+(\xi)}\left(\widehat{\tilde{\rho}}_0(\xi)+i\widehat{v}_0(\xi)\right),\\
		&\widehat{v}(t,\xi)\sim\frac{1}{2}e^{t\lambda_-(\xi)}\left(i\widehat{\tilde{\rho}}_0(\xi)+\widehat{v}_0(\xi)\right)+\frac{1}{2}e^{t\lambda_+(\xi)}\left(-i\widehat{\tilde{\rho}}_0(\xi)+\widehat{v}_0(\xi)\right).
	\end{split}
\end{equation}
Therefore, the low frequencies of $\tilde{\rho}$ and $v$ have a similar behavior. As $|e^{t\lambda_{\pm}(\xi)}|=e^{-\eta t|\xi|^2/2}$ applying Parseval's formula and gives
\begin{equation}
	\|(\dot{\Delta}_j\tilde{\rho},\dot{\Delta}_jv)(t)\|_{L^2}\le Ce^{-c\eta t 2^{2j}}\|(\dot{\Delta}_j\tilde{\rho}_0,\dot{\Delta}_jv_0)\|_{L^2}.
\end{equation}
And then, we can acquire the following proposition by putting the estimate for the above inequality and using Duhamel's formula:
\begin{proposition}\label{pro2.1}
	Let $s\in \mathbb{R}$. Assume that $(\tilde{\rho},\mathbf{u})$ satisfies
	\begin{equation}
		\begin{cases}
			\partial_t\tilde{\rho}+\bar{\rho}{\rm div}{\mathbf{u}}=N_1,\\
			\partial_t\mathbf{u}-\frac{\mu}{\bar{\rho}}\Delta \mathbf{u}-\frac{\mu+\mu'}{\bar{\rho}}\nabla{\rm div}\mathbf{u}+(\frac{P'(\bar{\rho})}{\bar{\rho}}+\gamma)\nabla \tilde{\rho}+\gamma\frac{\Delta}{1-\Delta}\nabla\tilde{\rho}=N_2.\\
		\end{cases}
	\end{equation}
Then we have for the low frequencies:
\begin{equation}
	\|(\tilde{\rho},\mathbf{u})\|^l_{\widetilde{L}^\infty_{t}(\dot{B}^s_{2,1})}+\eta\|(\tilde{\rho},\mathbf{u})\|^l_{L^1_{t}(\dot{B}^{s+2}_{2,1})}
\lesssim\|(\tilde{\rho}_0,\mathbf{u}_0)\|^l_{\dot{B}^s_{2,1}}+\|(N_1,N_2)\|^l_{L^1_{t}(\dot{B}^s_{2,1})}.
\end{equation}
\end{proposition}
Now we take $s=\frac{d}{2}-1$ in Proposition \ref{pro2.1} and obtain
\begin{equation}\label{2.34}
	\begin{split}
		&\|(\tilde{\rho},\mathbf{u})\|^l_{\widetilde{L}^\infty_t(\dot{B}^{\frac{d}{2}-1}_{2,1})}+(\alpha_3+\alpha_4)\|(\tilde{\rho},\mathbf{u})\|^l_{L^1_{t}(\dot{B}^{\frac{d}{2}+1}_{2,1})}\\
		\lesssim&\|(\tilde{\rho}_0,\mathbf{u}_0)\|^l_{\dot{B}^{\frac{d}{2}-1}_{2,1}}+\|N_1\|^l_{L^1_{t}(\dot{B}^{\frac{d}{2}-1}_{2,1})}+\|N_2\|^l_{L^1_{t}(\dot{B}^{\frac{d}{2}-1}_{2,1})}.
	\end{split}
\end{equation}
It follows from the expression of $\|N_1\|^l_{\dot{B}^{\frac{d}{2}-1}_{2,1}}$ that
\begin{equation}\label{2.35}
	\|N_1\|^l_{\dot{B}^{\frac{d}{2}-1}_{2,1}}\lesssim\|\tilde{\rho}{\rm div}\mathbf{u}\|^l_{\dot{B}^{\frac{d}{2}-1}_{2,1}}+\|\mathbf{u}\cdot\nabla \tilde{\rho}\|^l_{\dot{B}^{\frac{d}{2}-1}_{2,1}}.
\end{equation}
Bony decomposition gives
\begin{equation}\label{2.40}
	\tilde{\rho}{\rm div}\mathbf{u}=T_{\tilde{\rho}}{\rm div}\mathbf{u}+R(\tilde{\rho},{\rm div}\mathbf{u})+T_{{\rm div}\mathbf{u}}\tilde{\rho}.
\end{equation}
Thus, by using Proposition \ref{A.11}, it can be shown that
\begin{equation}\label{2.41}
	\begin{split}
		\|T_{\tilde{\rho}}{\rm div}\mathbf{u}+R(\tilde{\rho},{\rm div}\mathbf{u})\|^l_{\dot{B}^{\frac{d}{2}-1}_{2,1}}\lesssim&\|\tilde{\rho}\|_{\dot{B}^{\frac{d}{p}-1}_{p,1}}\|{\rm div}\mathbf{u}\|_{\dot{B}^{\frac{d}{p}}_{p,1}}
		\lesssim\|\tilde{\rho}\|_{\dot{B}^{\frac{d}{p}-1}_{p,1}}\|\mathbf{u}\|_{\dot{B}^{\frac{d}{p}+1}_{p,1}}
		\lesssim\mathscr{E}_\infty(t)\mathscr{E}_1(t),
	\end{split}
\end{equation}
and
\begin{equation}\label{2.42}
	\begin{split}
		\|T_{{\rm div}u}\tilde{\rho}\|^l_{\dot{B}^{\frac{d}{2}-1}_{2,1}}\lesssim&\|{\rm div}\mathbf{u}\|_{\dot{B}^{\frac{d}{p}-1}_{p,1}}\|\tilde{\rho}\|_{\dot{B}^{\frac{d}{p}}_{p,1}}
		\lesssim\|\mathbf{u}\|_{\dot{B}^{\frac{d}{p}}_{p,1}}\|\tilde{\rho}\|_{\dot{B}^{\frac{d}{p}}_{p,1}}\\
		\lesssim&\|\mathbf{u}\|^{\frac{1}{2}}_{\dot{B}^{\frac{d}{p}-1}_{p,1}}\|\mathbf{u}\|^{\frac{1}{2}}_{\dot{B}^{\frac{d}{p}+1}_{p,1}}\|\tilde{\rho}\|^{\frac{1}{2}}_{\dot{B}^{\frac{d}{p}-1}_{p,1}}\|\tilde{\rho}\|^{\frac{1}{2}}_{\dot{B}^{\frac{d}{p}+1}_{p,1}}
		\lesssim\mathscr{E}_\infty(t)\mathscr{E}_1(t).
	\end{split}
\end{equation}
From the above inequalities, we can acquire
\begin{equation}\label{2.43}
	\|\tilde{\rho}{\rm div}\mathbf{u}\|^l_{\dot{B}^{\frac{d}{2}-1}_{2,1}}\lesssim\mathscr{E}_\infty(t)\mathscr{E}_1(t).
\end{equation}
As for the term $\mathbf{u}\cdot\nabla \tilde{\rho}$, we handle this term by dividing it into two parts:
\begin{equation}
	\|\mathbf{u}\cdot\nabla \tilde{\rho}\|^l_{\dot{B}^{\frac{d}{2}-1}_{2,1}}\lesssim\|\mathbf{u}\cdot\nabla \tilde{\rho}^l\|^l_{\dot{B}^{\frac{d}{2}-1}_{2,1}}+\|\mathbf{u}\cdot\nabla \tilde{\rho}^h\|^l_{\dot{B}^{\frac{d}{2}-1}_{2,1}}.
\end{equation}
From Proposition \ref{A.7}, $\|\mathbf{u}\cdot\nabla \tilde{\rho}^l\|^l_{\dot{B}^{\frac{d}{2}-1}_{2,1}}$ can be bounded as
\begin{equation}
	\begin{split}
		\|\mathbf{u}\cdot\nabla \tilde{\rho}^l\|^l_{\dot{B}^{\frac{d}{2}-1}_{2,1}}\lesssim&\|\mathbf{u}\|_{\dot{B}^{\frac{d}{p}}_{p,1}}\|\nabla \tilde{\rho}^l\|_{\dot{B}^{\frac{d}{2}-1}_{2,1}}\lesssim\|\mathbf{u}\|_{\dot{B}^{\frac{d}{p}}_{p,1}}\|\tilde{\rho}^l\|_{\dot{B}^{\frac{d}{2}}_{2,1}}\\
		\lesssim&\||\mathbf{u}\|^{\frac{1}{2}}_{\dot{B}^{\frac{d}{p}-1}_{p,1}}\|\mathbf{u}\|^{\frac{1}{2}}_{\dot{B}^{\frac{d}{p}+1}_{p,1}}\|\tilde{\rho}^l\|^{\frac{1}{2}}_{\dot{B}^{\frac{d}{2}-1}_{2,1}}\|\tilde{\rho}^l\|^{\frac{1}{2}}_{\dot{B}^{\frac{d}{2}+1}_{2,1}}
		\lesssim\mathscr{E}_\infty(t)\mathscr{E}_1(t).
	\end{split}
\end{equation}
For the term $\|\mathbf{u}\cdot\nabla \tilde{\rho}^h\|^l_{\dot{B}^{\frac{d}{2}-1}_{2,1}}$, we also use Bony decomposition to have
\begin{equation}
	\mathbf{u}\cdot\nabla \tilde{\rho}^h=T_{\mathbf{u}}\nabla \tilde{\rho}^h+R(\mathbf{u},\nabla \tilde{\rho}^h)+T_{\nabla \tilde{\rho}^h}\mathbf{u}.
\end{equation}
Then Proposition \ref{A.11} leads to
\begin{equation*}\label{2.37}
	\begin{split}
		\|T_{\mathbf{u}}\nabla \tilde{\rho}^h\|^l_{\dot{B}^{\frac{d}{2}-1}_{2,1}}\lesssim&\|T_{\mathbf{u}}\nabla \tilde{\rho}^h\|^l_{\dot{B}^{\frac{d}{2}-2}_{2,1}}\lesssim\|\mathbf{u}\|_{\dot{B}^{\frac{d}{p}-1}_{p,1}}\|\nabla\tilde{\rho}^h\|_{\dot{B}^{\frac{d}{p}-1}_{p,1}}
		\lesssim\|\mathbf{u}\|_{\dot{B}^{\frac{d}{p}-1}_{p,1}}\|\tilde{\rho}^h\|_{\dot{B}^{\frac{d}{p}}_{p,1}}
		\lesssim\mathscr{E}_\infty(t)\mathscr{E}_1(t),
	\end{split}
\end{equation*}
and
\begin{equation*}\label{2.38}
	\begin{split}
		\|R(\mathbf{u},\nabla \tilde{\rho}^h)+T_{\nabla \tilde{\rho}^h}\mathbf{u}\|^l_{\dot{B}^{\frac{d}{2}-1}_{2,1}}\lesssim&\|\nabla \tilde{\rho}^h\|_{\dot{B}^{\frac{d}{p}-1}_{p,1}}\|\mathbf{u}\|_{\dot{B}^{\frac{d}{p}}_{p,1}}
		\lesssim\|\tilde{\rho}^h\|_{\dot{B}^{\frac{d}{p}}_{p,1}}\|\mathbf{u}\|_{\dot{B}^{\frac{d}{p}}_{p,1}}\\
		\lesssim&\|\tilde{\rho}^h\|^{\frac{1}{2}}_{\dot{B}^{\frac{d}{p}}_{p,1}}\|\tilde{\rho}^h\|^{\frac{1}{2}}_{\dot{B}^{\frac{d}{p}}_{p,1}}\|\mathbf{u}\|^{\frac{1}{2}}_{\dot{B}^{\frac{d}{p}-1}_{p,1}}\|\mathbf{u}\|^{\frac{1}{2}}_{\dot{B}^{\frac{d}{p}+1}_{p,1}}
		\lesssim\mathscr{E}_\infty(t)\mathscr{E}_1(t).
	\end{split}
\end{equation*}
Hence, we have
\begin{equation}\label{2.39}
	\|\mathbf{u}\cdot\nabla \tilde{\rho}^h\|^l_{\dot{B}^{\frac{d}{2}-1}_{2,1}}\lesssim\mathscr{E}_\infty(t)\mathscr{E}_1(t).
\end{equation}
As the term $\|N_2\|^l_{\dot{B}^{\frac{d}{2}-1}_{2,1}}$ has been estimated in \eqref{2.25}, we safely arrive at the following result
\begin{equation}\label{2.45}
	\begin{split}		\|(\tilde{\rho},\mathbf{u})\|^l_{\widetilde{L}^\infty_t(\dot{B}^{\frac{d}{2}-1}_{2,1})}+(\alpha_3+\alpha_4)\|(\tilde{\rho},\mathbf{u})\|^l_{L^1_{t}(\dot{B}^{\frac{d}{2}+1}_{2,1})}
		\lesssim\|(\tilde{\rho}_0,\mathbf{u}_0)\|^l_{\dot{B}^{\frac{d}{2}-1}_{2,1}}+\int_{0}^{t}\mathscr{E}_\infty(\tau)\mathscr{E}_1(\tau)d\tau.
	\end{split}
\end{equation}

\subsection{The  estimate of $\mathbb{P}u$ in the high frequency.}
\quad\quad Next, we are ready to estimate $\mathbb{P}\mathbf{u}$ in the high frequency. The same as \eqref{2.4} in Subsection 2.1, we take the operator $\mathbb{P}$ again to the second equation in \eqref{2.1} and get
\begin{equation*}
	\partial_t\mathbb{P}\mathbf{u}-\alpha_3\Delta\mathbb{P}\mathbf{u}=\mathbb{P}N_2.
\end{equation*}
Applying $\dot{\Delta}_j$ to the above equation,  and then taking the $L^2$ inner product with $|\dot{\Delta}_j\mathbb{P}\mathbf{u}|^{p-2}\dot{\Delta}_j\mathbb{P}\mathbf{u}$, it holds by using H$\ddot{o}$lder's inequality that
\begin{equation}\label{2.46}
	\frac{1}{p}\frac{d}{dt}\|\dot{\Delta}_j\mathbb{P}\mathbf{u}\|^p_{L^p}+c\alpha_32^{2j}\|\dot{\Delta}_j\mathbb{P}\mathbf{u}\|^p_{L^p}\lesssim\|\dot{\Delta}_j\mathbb{P}N_2\|_{L^p}\|\dot{\Delta}_j\mathbb{P}\mathbf{u}\|^{p-1}_{L^p},
\end{equation}
in which the term $c2^{2j}\|\dot{\Delta}_j\mathbb{P}\mathbf{u}\|^p_{L^p}$ is arising from Proposition \ref{A.6}: there exists a positive constant $c$ so that
\begin{equation*}
	-\int_{\mathbb{R}^d}\Delta\dot{\Delta}_j\mathbb{P}\mathbf{u}\cdot|\dot{\Delta}_j\mathbb{P}\mathbf{u}|^{p-2}\dot{\Delta}_j\mathbb{P}\mathbf{u}dx=(p-1)\int_{\mathbb{R}^d}|\nabla\dot{\Delta}_j\mathbb{P}\mathbf{u}|^2|\dot{\Delta}_j\mathbb{P}\mathbf{u}|^{p-2}dx\ge c2^{2j}\|\dot{\Delta}_j\mathbb{P}\mathbf{u}\|^p_{L^p}.
\end{equation*}
Afterwards, multiplying by $1/\|\dot{\Delta}_j\mathbb{P}\mathbf{u}\|^{p-1}_{L^p}2^{j(\frac{d}{p}-1)}$ on both hand side of \eqref{2.46}, integrating the resultant inequality from 0 to $t$, we can obtain by summing up about $j\ge j_0$ that
\begin{equation}\label{2.47}
	\|\mathbb{P}\mathbf{u}\|^h_{\widetilde{L}^\infty_{t}(\dot{B}^{\frac{d}{p}-1}_{p,1})}+\alpha_3\|\mathbb{P}\mathbf{u}\|^h_{L^1_{t}(\dot{B}^{\frac{d}{p}+1}_{p,1})}
	\lesssim
	\|\mathbb{P}\mathbf{u}_0\|^h_{\dot{B}^{\frac{d}{p}-1}_{p,1}}+\|N_2\|^h_{L^1_{t}(\dot{B}^{\frac{d}{p}-1}_{p,1})}.
\end{equation}
In terms of $\|{N_2}\|^h_{\dot{B}^{\frac{d}{p}-1}_{p,1}}$, we have
\begin{equation}\label{2.48}
	\begin{split}
		\|N_2\|^h_{\dot{B}^{\frac{d}{p}-1}_{p,1}}\lesssim&\|\mathbf{u}\cdot\nabla \mathbf{u}\|^h_{\dot{B}^{\frac{d}{p}-1}_{p,1}}+\|F(\tilde{\rho})\nabla \tilde{\rho}\|^h_{\dot{B}^{\frac{d}{p}-1}_{p,1}}+\|k(\tilde{\rho})\Delta \mathbf{u}\|^h_{\dot{B}^{\frac{d}{p}+1}_{p,1}}+\|k(\tilde{\rho})\nabla{\rm div}\mathbf{u}\|^h_{\dot{B}^{\frac{d}{p}-1}_{p,1}}.
	\end{split}
\end{equation}
Now, we begin to estimate the right-hand terms in \eqref{2.48} by using \eqref{a.7} of Proposition \ref{A.7}. For the first term $\|\mathbf{u}\cdot\nabla \mathbf{u}\|^h_{\dot{B}^{\frac{d}{p}-1}_{p,1}}$, we can get
\begin{equation}\label{2.49}
	\!\begin{split}
		\|\mathbf{u}\!\cdot\!\nabla \mathbf{u}\|^h_{\dot{B}^{\frac{d}{p}-1}_{p,1}}\lesssim&\|\mathbf{u}\|_{\dot{B}^{\frac{d}{p}}_{p,1}}\!\|\nabla \mathbf{u}\|_{\dot{B}^{\frac{d}{p}-1}_{p,1}}
		\!\lesssim\!\|\mathbf{u}\|_{\dot{B}^{\frac{d}{p}}_{p,1}}\!\|\mathbf{u}\|_{\dot{B}^{\frac{d}{p}}_{p,1}}
		\!\lesssim\!\|\mathbf{u}\|_{\dot{B}^{\frac{d}{p}-1}_{p,1}}\|\mathbf{u}\|_{\dot{B}^{\frac{d}{p}+1}_{p,1}}
		\lesssim\mathscr{E}_\infty(t)\mathscr{E}_1(t).
	\end{split}
\end{equation}
For the second term $\|F(\tilde{\rho})\nabla \tilde{\rho}\|^h_{\dot{B}^{\frac{d}{p}-1}_{p,1}}$, we have
\begin{equation}\label{2.50}
	\begin{split}
		\|F(\tilde{\rho})\nabla \tilde{\rho}\|^h_{\dot{B}^{\frac{d}{p}-1}_{p,1}}\lesssim&\|F(\tilde{\rho})\nabla \tilde{\rho}^l\|_{\dot{B}^{\frac{d}{p}-1}_{p,1}}+\|F(\tilde{\rho})\nabla \tilde{\rho}^h\|_{\dot{B}^{\frac{d}{p}-1}_{p,1}}\\
		\lesssim&\|\nabla \tilde{\rho}^l\|_{\dot{B}^{\frac{d}{p}}_{p,1}}\|F(\tilde{\rho})\|_{\dot{B}^{\frac{d}{p}-1}_{p,1}}+\|F(\tilde{\rho})\|_{\dot{B}^{\frac{d}{p}}_{p,1}}\|\nabla \tilde{\rho}^l\|_{\dot{B}^{\frac{d}{p}-1}_{p,1}}\\
		\lesssim&\|\tilde{\rho}^l\|_{\dot{B}^{\frac{d}{p}+1}_{p,1}}(1+\|\tilde{\rho}\|_{\dot{B}^{\frac{d}{p}}_{p,1}})\|\tilde{\rho}\|_{\dot{B}^{\frac{d}{p}-1}_{p,1}}+\|\tilde{\rho}\|_{\dot{B}^{\frac{d}{p}}_{p,1}}\|\tilde{\rho}^l\|_{\dot{B}^{\frac{d}{p}}_{p,1}}\\
		\lesssim&(1+\mathscr{E}_\infty(t))\mathscr{E}_\infty(t)\mathscr{E}_1(t)+\mathscr{E}_\infty(t)\mathscr{E}_1(t)
		\lesssim\mathscr{E}_\infty(t)\mathscr{E}_1(t).
	\end{split}
\end{equation}
For the third and last term $\|k(\tilde{\rho})\Delta \mathbf{u}\|^h_{\dot{B}^{\frac{d}{p}-1}_{p,1}}$, it can be bounded as
\begin{equation}\label{2.51}
	\begin{split}
		\|k(\tilde{\rho})\Delta \mathbf{u}\|^h_{\dot{B}^{\frac{d}{p}-1}_{p,1}}
		\lesssim\|k(\tilde{\rho})\|_{\dot{B}^{\frac{d}{p}}_{p,1}}\|\Delta \mathbf{u}\|_{\dot{B}^{\frac{d}{p}-1}_{p,1}}
		\lesssim\|\tilde{\rho}\|_{\dot{B}^{\frac{d}{p}}_{p,1}}\|\mathbf{u}\|_{\dot{B}^{\frac{d}{p}+1}_{p,1}}
		\lesssim\mathscr{E}_\infty(t)\mathscr{E}_1(t).
	\end{split}
\end{equation}
The last term $\|k(\tilde{\rho})\nabla{\rm div}\mathbf{u}\|^h_{\dot{B}^{\frac{d}{p}-1}_{p,1}}$ is similar to $\|k(\tilde{\rho})\Delta \mathbf{u}\|^h_{\dot{B}^{\frac{d}{p}-1}_{p,1}}$, so we omit the details. Plugging \eqref{2.49}-\eqref{2.51} into \eqref{2.48} gives that
\begin{equation}\label{2.52}
	\|N_2\|^h_{\dot{B}^{\frac{d}{p}-1}_{p,1}}\lesssim\mathscr{E}_\infty(t)\mathscr{E}_1(t).
\end{equation}
Ultimately, we can get
\begin{equation}\label{2.53}
	\begin{split}
		\|\mathbb{P}\mathbf{u}\|^h_{\widetilde{L}^\infty_{t}(\dot{B}^{\frac{d}{p}-1}_{p,1})}+\alpha_3\|\mathbb{P}\mathbf{u}\|^h_{L^1_{t}(\dot{B}^{\frac{d}{p}+1}_{p,1})}
		\lesssim \|\mathbb{P}\mathbf{u}_0\|^h_{\dot{B}^{\frac{d}{p}-1}_{p,1}}+\int_{0}^{t}\mathscr{E}_\infty(\tau)\mathscr{E}_1(\tau)d\tau.
	\end{split}
\end{equation}

\subsection{The estimate of $(\tilde{\rho},\mathbb{Q}\mathbf{u})$ in the high frequency}
\quad\quad In this subsection, we are going to estimate the term $(\tilde{\rho},\mathbb{Q}\mathbf{u})$ in the high frequency. On the one hand, in order to get the results we want, we use a quantity from \cite{chen4,haspot} that
\begin{equation}\label{2.64}
	G:=\mathbb{Q}\mathbf{u}-\frac{\alpha_1}{\alpha_3+\alpha_4}\Delta^{-1}\nabla \tilde{\rho},
\end{equation}
and then use the second equation of \eqref{2.1} to get
\begin{equation}\label{2.65}
	\begin{split}
		&\partial_tG-(\alpha_3+\alpha_4)\Delta G\\
		=&\mathbb{Q}N_2-\frac{\alpha_1}{\alpha_3+\alpha_4}\Delta^{-1}\nabla N_1+\frac{\alpha_1\alpha_2}{\alpha_3+\alpha_4}G+\frac{\alpha_1^2\alpha_2}{(\alpha_3+\alpha_4)^2}\Delta^{-1}\nabla\tilde{\rho}-\gamma\frac{\nabla}{1-\Delta}\tilde{\rho}.
	\end{split}
\end{equation}
After that, we can gain by a standard energy argument that
\begin{equation}\label{2.66}
	\begin{split}
		&\|G\|^h_{\widetilde{L}^\infty_{t}(\dot{B}^{\frac{d}{p}-1}_{p,1})}+\|G\|^h_{L^1_t(\dot{B}^{\frac{d}{p}+1}_{p,1})}\\
		\lesssim&\|G_0\|^h_{\dot{B}^{\frac{d}{p}-1}_{p,1}}+\|\mathbb{Q}N_2\|^h_{L^1_{t}(\dot{B}^{\frac{d}{p}-1}_{p,1})}+\|\Delta^{-1}\nabla N_1\|^h_{L^1_{t}(\dot{B}^{\frac{d}{p}-1}_{p,1})}\\
		&+\|G\|^h_{L^1_{t}(\dot{B}^{\frac{d}{p}-1}_{p,1})}+\|\Delta^{-1}\nabla \tilde{\rho}\|^h_{L^1_{t}(\dot{B}^{\frac{d}{p}-1}_{p,1})}+\|\frac{\nabla}{1-\Delta}\tilde{\rho}\|^h_{L^1_{t}(\dot{B}^{\frac{d}{p}-1}_{p,1})}\\
		\lesssim&\|G_0\|^h_{\dot{B}^{\frac{d}{p}-1}_{p,1}}+\|N_2\|^h_{L^1_{t}(\dot{B}^{\frac{d}{p}-1}_{p,1})}+\|N_1\|^h_{L^1_{t}(\dot{B}^{\frac{d}{p}-2}_{p,1})}\\
		&+\|G\|^h_{L^1_{t}(\dot{B}^{\frac{d}{p}-1}_{p,1})}+\|\tilde{\rho}\|^h_{L^1_{t}(\dot{B}^{\frac{d}{p}-2}_{p,1})}.
	\end{split}
\end{equation}
Here we have used the following relation in high frequency:
\begin{equation*}	\big\|\frac{\nabla}{1-\Delta}\tilde{\rho}\big\|^h_{\dot{B}^{\frac{d}{p}-1}_{p,1}}\lesssim\|\tilde{\rho}\|^h_{\dot{B}^{\frac{d}{p}-2}_{p,1}},
\end{equation*}
which can be also obtained by using the classical estimate on the Bessel potential $(1-\Delta)^{-1}$ for any function $f\in\dot{B}_{p,r}^s$:
\begin{equation}
\|(1-\Delta)^{-1}f\|_{\dot{B}_{p,r}^s\cap\dot{B}_{p,r}^{s+2}}\lesssim\|f\|_{\dot{B}_{p,r}^s},\ s\in\mathbb{R},\ p,r\in[1,\infty].
\end{equation}
On the other hand, applying the divergence operator to \eqref{2.64} and summing up with the first equation of \eqref{2.1}, we have
\begin{equation*}
	\partial_t\tilde{\rho}+\frac{\alpha_1\alpha_2}{\alpha_3+\alpha_4}\tilde{\rho}+\mathbf{u}\cdot\nabla \tilde{\rho}=-\tilde{\rho}{\rm div}\mathbf{u}-\alpha_2{\rm div}G.
\end{equation*}
Taking $\dot{\Delta}_j$ to the above equation and using the commutator's argument to get
\begin{equation}\label{2.67}
	\partial_t\dot{\Delta}_j\tilde{\rho}+\frac{\alpha_1\alpha_2}{\alpha_3+\alpha_4}\dot{\Delta}_j\tilde{\rho}+\mathbf{u}\cdot\nabla
	\dot{\Delta}_j\tilde{\rho}+[\dot{\Delta}_j, \mathbf{u}\cdot\nabla]\tilde{\rho}
	=-\dot{\Delta}_j(\tilde{\rho}{\rm div}\mathbf{u})-\alpha_2\dot{\Delta}_j{\rm div}G.
\end{equation}
Similarly, following a standard energy argument, we get
\begin{equation}\label{2.68}
	\begin{split}
		\|\tilde{\rho}\|^h_{\widetilde{L}^\infty_{t}(\dot{B}^{\frac{d}{p}}_{p,1})}\!+\!\frac{\alpha_1\alpha_2}{\alpha_3+\alpha_4}&\|\tilde{\rho}\|^h_{L^1_t(\dot{B}^{\frac{d}{p}}_{p,1})}
		\lesssim\|\tilde{\rho}_0\|^h_{\dot{B}^{\frac{d}{p}}_{p,1}}\!+\!\|\tilde{\rho}{\rm div}\mathbf{u}\|^h_{L^1_{t}(\dot{B}^{\frac{d}{p}}_{p,1})}+\|{\rm div}G\|^h_{L^1_{t}(\dot{B}^{\frac{d}{p}}_{p,1})}\\
		& +\!\int_{0}^{t}\!\|{\rm div}\mathbf{u}\|_{L^\infty}\|\tilde{\rho}\|^h_{\dot{B}^{\frac{d}{p}}_{p,1}}\!d\tau\!+\!\int_{0}^{t}\!\sum\limits_{j\ge j_0}2^{\frac{d}{p}j}\|[\dot{\Delta}_j, \mathbf{u}\cdot\nabla]\tilde{\rho}\|_{L^2}d\tau.
	\end{split}
\end{equation}
By virtue of the high frequency cut-off, we have
\begin{equation}\label{2.69}
	\|G\|^h_{L^1_t(\dot{B}^{\frac{d}{p}-1}_{p,1})}\lesssim2^{-2j_0}\|G\|^h_{L^1_t(\dot{B}^{\frac{d}{p}+1}_{p,1})}\ \ {\rm and}\ \ \|\tilde{\rho}\|^h_{L^1_t(\dot{B}^{\frac{d}{p}-2}_{p,1})}\lesssim2^{-2j_0}\|\tilde{\rho}\|^h_{L^1_t(\dot{B}^{\frac{d}{p}}_{p,1})}.
\end{equation}
Choosing $j_0$ large enough such that $\|G\|^h_{L^1_t(\dot{B}^{\frac{d}{p}-1}_{p,1})}$, then $\|\tilde{\rho}\|^h_{L^1_t(\dot{B}^{\frac{d}{p}-2}_{p,1})}$ can be absorbed by the left hand side of \eqref{2.66} and \eqref{2.68}. Therefore, summing up \eqref{2.66} and \eqref{2.68}, we finally get
\begin{equation}\label{2.70}
	\begin{split}
		&\|G\|^h_{\widetilde{L}^\infty_{t}(\dot{B}^{\frac{d}{p}-1}_{p,1})}+\|G\|^h_{L^1_t(\dot{B}^{\frac{d}{p}+1}_{p,1})}+\|\tilde{\rho}\|^h_{\widetilde{L}^\infty_{t}(\dot{B}^{\frac{d}{p}}_{p,1})}+\frac{\alpha_1\alpha_2}{\alpha_3+\alpha_4}\|\tilde{\rho}\|^h_{L^1_t(\dot{B}^{\frac{d}{p}}_{p,1})}\\
		\lesssim&\|G_0\|^h_{\dot{B}^{\frac{d}{p}-1}_{p,1}}+\|\tilde{\rho}_0\|^h_{\dot{B}^{\frac{d}{p}}_{p,1}}+\|N_1\|^h_{L^1_{t}(\dot{B}^{\frac{d}{p}-2}_{p,1})}+\|N_2\|^h_{L^1_{t}(\dot{B}^{\frac{d}{p}-1}_{p,1})}\\
		&+\int_{0}^{t}\|{\rm div}\mathbf{u}\|_{L^\infty}\|\tilde{\rho}\|^h_{\dot{B}^{\frac{d}{p}}_{p,1}}d\tau+\int_{0}^{t}\sum\limits_{j\ge j_0}2^{\frac{d}{p}j}\|[\dot{\Delta}_j, \mathbf{u}\cdot\nabla]\tilde{\rho}\|_{L^2}d\tau
		+\int_{0}^{t}\|\tilde{\rho}{\rm div}\mathbf{u}\|^h_{\dot{B}^{\frac{d}{p}}_{p,1}}d\tau.
	\end{split}
\end{equation}
From the expression of $\|N_1\|^h_{\dot{B}^{\frac{d}{p}-2}_{p,1}}$, one has
\begin{equation}\label{2.71}
	\|N_1\|^h_{\dot{B}^{\frac{d}{p}-2}_{p,1}}\lesssim\|\tilde{\rho}{\rm div}\mathbf{u}\|^h_{\dot{B}^{\frac{d}{p}-2}_{p,1}}+\|\mathbf{u}\cdot\nabla\tilde{\rho}\|^h_{\dot{B}^{\frac{d}{p}-2}_{p,1}}.
\end{equation}
Following \eqref{a.8} of Proposition \ref{A.7}, it holds
\begin{equation}\label{2.72}
	\begin{split}
		\|\tilde{\rho}{\rm div}\mathbf{u}\|^h_{\dot{B}^{\frac{d}{p}-2}_{p,1}}\lesssim&\|\tilde{\rho}{\rm div}\mathbf{u}\|^h_{\dot{B}^{\frac{d}{p}}_{p,1}}\lesssim\|\tilde{\rho}\|_{\dot{B}^{\frac{d}{p}}_{p,1}}\|{\rm div}\mathbf{u}\|_{\dot{B}^{\frac{d}{p}}_{p,1}}
		\lesssim\|\tilde{\rho}\|_{\dot{B}^{\frac{d}{p}}_{p,1}}\|\mathbf{u}\|_{\dot{B}^{\frac{d}{p}+1}_{p,1}}
		\lesssim\mathscr{E}_\infty(t)\mathscr{E}_1(t),
	\end{split}
\end{equation}
and
\begin{equation}\label{2.73}
	\begin{split}
		\|\mathbf{u}\cdot\nabla\tilde{\rho}\|^h_{\dot{B}^{\frac{d}{p}-2}_{p,1}}\lesssim&\|\mathbf{u}\cdot\nabla\tilde{\rho}\|^h_{\dot{B}^{\frac{d}{p}-1}_{p,1}}\lesssim\|\mathbf{u}\|_{\dot{B}^{\frac{d}{p}}_{p,1}}\|\nabla\tilde{\rho}\|_{\dot{B}^{\frac{d}{p}-1}_{p,1}}
		\lesssim\|\mathbf{u}\|_{\dot{B}^{\frac{d}{p}}_{p,1}}\|\tilde{\rho}\|_{\dot{B}^{\frac{d}{p}}_{p,1}}\\
		\lesssim&\|\mathbf{u}\|_{\dot{B}^{\frac{d}{p}-1}_{p,1}}^{\frac{1}{2}}\|\mathbf{u}\|_{\dot{B}^{\frac{d}{p}+1}_{p,1}}^{\frac{1}{2}}\|\tilde{\rho}\|_{\dot{B}^{\frac{d}{p}-1}_{p,1}}^{\frac{1}{2}}\|\tilde{\rho}\|_{\dot{B}^{\frac{d}{p}+1}_{p,1}}^{\frac{1}{2}}\\
		\lesssim&\mathscr{E}_\infty(t)\mathscr{E}_1(t).
	\end{split}
\end{equation}
So, plugging \eqref{2.72} and \eqref{2.73} into \eqref{2.71}, $\|N_1\|^h_{\dot{B}^{\frac{d}{p}-2}_{p,1}}$ can be bounded as
\begin{equation}\label{2.74}
	\|N_1\|^h_{\dot{B}^{\frac{d}{p}-2}_{p,1}}\lesssim\mathscr{E}_\infty(t)\mathscr{E}_1(t).
\end{equation}
Besides, we have already estimated the term $\|N_2\|^h_{\dot{B}^{\frac{d}{p}-1}_{p,1}}$ completely in \eqref{2.52}. Now, using \eqref{a.17} of Lemma \ref{A.12}, we can get
\begin{equation}\label{2.75}
	\begin{split}
		&\|{\rm div}\mathbf{u}\|_{L^\infty}\|\tilde{\rho}\|^h_{\dot{B}^{\frac{d}{p}}_{p,1}}+\sum\limits_{j\ge j_0}2^{\frac{d}{p}j}\|[\dot{\Delta}_j, \mathbf{u}\cdot\nabla]\tilde{\rho}\|_{L^2}\\
		\lesssim&\|{\rm div}\mathbf{u}\|_{\dot{B}^{\frac{d}{p}}_{p,1}}\|\tilde{\rho}\|^h_{\dot{B}^{\frac{d}{p}}_{p,1}}+\|\nabla \mathbf{u}\|_{\dot{B}^{\frac{d}{p}}_{p,1}}\|\tilde{\rho}\|_{\dot{B}^{\frac{d}{p}}_{p,1}}		\lesssim\|\mathbf{u}\|_{\dot{B}^{\frac{d}{p}+1}_{p,1}}\|\tilde{\rho}\|^h_{\dot{B}^{\frac{d}{p}}_{p,1}}+\|\mathbf{u}\|_{\dot{B}^{\frac{d}{p}+1}_{p,1}}\|\tilde{\rho}\|_{\dot{B}^{\frac{d}{p}}_{p,1}}\\
		\lesssim&\mathscr{E}_\infty(t)\mathscr{E}_1(t).
	\end{split}
\end{equation}
From \eqref{a.7} of Proposition \ref{A.7}, $\|\tilde{\rho}{\rm div}\mathbf{u}\|^h_{\dot{B}^{\frac{d}{p}}_{p,1}}$ can be bounded as
\begin{equation}\label{2.76}
	\begin{split}
		\|\tilde{\rho}{\rm div}\mathbf{u}\|^h_{\dot{B}^{\frac{d}{p}}_{p,1}}\lesssim\|\tilde{\rho}\|_{\dot{B}^{\frac{d}{p}}_{p,1}}\|{\rm div}\mathbf{u}\|_{\dot{B}^{\frac{d}{p}}_{p,1}}\lesssim\|\tilde{\rho}\|_{\dot{B}^{\frac{d}{p}}_{p,1}}\|\mathbf{u}\|_{\dot{B}^{\frac{d}{p}+1}_{p,1}}
		\lesssim\mathscr{E}_\infty(t)\mathscr{E}_1(t).
	\end{split}
\end{equation}
In aid of all the above inequalities, we have
\begin{equation}\label{2.79}
	\begin{split}
		&\|G\|^h_{\widetilde{L}^\infty_{t}(\dot{B}^{\frac{d}{p}-1}_{p,1})}+\|G\|^h_{L^1_t(\dot{B}^{\frac{d}{p}+1}_{p,1})}+\|\tilde{\rho}\|^h_{\widetilde{L}^\infty_{t}(\dot{B}^{\frac{d}{p}}_{p,1})}+\frac{\alpha_1\alpha_2}{\alpha_3+\alpha_4}\|\tilde{\rho}\|^h_{L^1_t(\dot{B}^{\frac{d}{p}}_{p,1})}\\
		\lesssim&\|G_0\|^h_{\dot{B}^{\frac{d}{p}-1}_{p,1}}+\|\tilde{\rho}_0\|^h_{\dot{B}^{\frac{d}{p}}_{p,1}}+\int_{0}^{t}\mathscr{E}_\infty(\tau)\mathscr{E}_1(\tau)d\tau.
	\end{split}
\end{equation}
Owing to $G:=\mathbb{Q}\mathbf{u}-\frac{\alpha_1}{\alpha_3+\alpha_4}\Delta^{-1}\nabla \tilde{\rho}$, we can simply get
\begin{equation}\label{2.80}
	\begin{split}
		\|\mathbb{Q}\mathbf{u}\|^h_{\widetilde{L}^\infty_{t}(\dot{B}^{\frac{d}{p}-1}_{p,1})}\lesssim&\|G\|^h_{\widetilde{L}^\infty_{t}(\dot{B}^{\frac{d}{p}-1}_{p,1})}
+\|\Delta^{-1}\nabla \tilde{\rho}\|^h_{\widetilde{L}^\infty_{t}(\dot{B}^{\frac{d}{p}-1}_{p,1})}
		\lesssim\|G\|^h_{\widetilde{L}^\infty_{t}(\dot{B}^{\frac{d}{p}-1}_{p,1})}+\|\tilde{\rho}\|^h_{\widetilde{L}^\infty_{t}(\dot{B}^{\frac{d}{p}}_{p,1})},
	\end{split}
\end{equation}
\begin{equation}\label{2.81}
	\begin{split}
		\|\mathbb{Q}\mathbf{u}\|^h_{L^1_t(\dot{B}^{\frac{d}{p}+1}_{p,1})}\lesssim&\|G\|^h_{L^1_t(\dot{B}^{\frac{d}{p}+1}_{p,1})}+\|\Delta^{-1}\nabla \tilde{\rho}\|^h_{L^1_t(\dot{B}^{\frac{d}{p}+1}_{p,1})}
		\lesssim\|G\|^h_{L^1_t(\dot{B}^{\frac{d}{p}+1}_{p,1})}+\|\tilde{\rho}\|^h_{L^1_t(\dot{B}^{\frac{d}{p}}_{p,1})},
	\end{split}
\end{equation}
and
\begin{equation}\label{2.82}
	\begin{split}
		\|G_0\|^h_{\dot{B}^{\frac{d}{p}-1}_{p,1}}\lesssim&\|\mathbb{Q}\mathbf{u}_0\|^h_{\dot{B}^{\frac{d}{p}-1}_{p,1}}+\|\Delta^{-1}\nabla \tilde{\rho}_0\|^h_{\dot{B}^{\frac{d}{p}+1}_{p,1}}
		\lesssim\|\mathbb{Q}\mathbf{u}_0\|^h_{\dot{B}^{\frac{d}{p}-1}_{p,1}}+\|\tilde{\rho}_0^h\|_{\dot{B}^{\frac{d}{p}}_{p,1}}.
	\end{split}
\end{equation}
Plugging \eqref{2.80}-\eqref{2.82} into \eqref{2.79}, we eventually arrive at
\begin{equation}\label{2.83}
	\begin{split}
		&\|\mathbb{Q}\mathbf{u}\|^h_{\widetilde{L}^\infty_{t}(\dot{B}^{\frac{d}{p}-1}_{p,1})}+\|\mathbb{Q}\mathbf{u}\|^h_{L^1_t(\dot{B}^{\frac{d}{p}+1}_{p,1})}+\|\tilde{\rho}\|^h_{\widetilde{L}^\infty_{t}(\dot{B}^{\frac{d}{p}}_{p,1})}+\frac{\alpha_1\alpha_2}{\alpha_3+\alpha_4}\|\tilde{\rho}\|^h_{L^1_t(\dot{B}^{\frac{d}{p}}_{p,1})}\\
		\lesssim&\|\mathbb{Q}\mathbf{u}_0\|^h_{\dot{B}^{\frac{d}{p}-1}_{p,1}}+\|\tilde{\rho}_0\|^h_{\dot{B}^{\frac{d}{p}}_{p,1}}+\int_{0}^{t}\mathscr{E}_\infty(\tau)\mathscr{E}_1(\tau)d\tau.
	\end{split}
\end{equation}

\vspace{3mm}
			
\noindent{\textbf{The proof of Theorem \ref{Theorem 1.1}}}:
First, define
\begin{equation}
\begin{array}{rl}
&\mathcal{X}(t)\triangleq\sup\limits_{t>0}\mathscr{E}_\infty(t)+\int_0^t\mathscr{E}_1(\tau)d\tau,\\[2mm]
&\displaystyle\mathcal{X}(0)\triangleq\|(\tilde{\rho}_0^l,\mathbf{u}_0^l)\|_{\dot{B}_{2,1}^{\frac{d}{2}-1}}+\|\tilde{\rho}^h\|_{\dot{B}_{p,1}^{\frac{d}{p}}}+\|\mathbf{u}_0^h\|_{\dot{B}_{p,1}^{\frac{d}{p}-1}}.
\end{array}
\end{equation}
Then, by virtue of the estimates in Sections 2.1-2.4, one knows that there exists a constant $C>0$ independent of $t$ such that
\begin{equation}\label{2.85}
\mathcal{X}(t)\leq C\mathcal{X}(0)+C\mathcal{X}(t)^2.
\end{equation}

On the other hand, the local existence has been achieved in Chikami \cite{chikami1}. Using the setting of initial data in Theorem \ref{Theorem 1.1} and the local existence, one can get for some $C>0$ that
\begin{equation}\label{2.86}
\mathcal{X}(t)\leq 2Cc_0,\ \ \forall t\in[0,T].
\end{equation}
Then, a standard continuation argument suffices to yield the global existence in Theorem \ref{Theorem 1.1}.

\section{Decay rate}
\quad \quad In this section, we shall derive the decay rate of the solution constructed in Section 2.
From Section 2, we can get the following inequality:
\begin{equation}\label{4.28}
	\begin{split}
	&\frac{d}{dt}(\|(\tilde{\rho},\mathbf{u})\|^l_{\dot{B}^{\frac{d}{2}-1}_{2,1}}+\|\tilde{\rho}\|^h_{\dot{B}^{\frac{d}{p}}_{p,1}}+\|\mathbf{u}\|^h_{\dot{B}^{\frac{d}{p}-1}_{p,1}})+\|(\tilde{\rho},\mathbf{u})\|^l_{\dot{B}^{\frac{d}{2}+1}_{2,1}}+\|\tilde{\rho}\|^h_{\dot{B}^{\frac{d}{p}}_{p,1}}+\|\mathbf{u}\|^h_{\dot{B}^{\frac{d}{p}+1}_{p,1}}\\
	\lesssim&(\|(\tilde{\rho},\mathbf{u})\|^l_{\dot{B}^{\frac{d}{2}-1}_{2,1}}+\|\tilde{\rho}\|^h_{\dot{B}^{\frac{d}{p}}_{p,1}}+\|\mathbf{u}\|^h_{\dot{B}^{\frac{d}{p}-1}_{p,1}})(\|(\tilde{\rho},\mathbf{u})\|^l_{\dot{B}^{\frac{d}{2}+1}_{2,1}}+\|\tilde{\rho}\|^h_{\dot{B}^{\frac{d}{p}}_{p,1}}+\|\mathbf{u}\|^h_{\dot{B}^{\frac{d}{p}+1}_{p,1}}).
\end{split}
\end{equation}
By the proof of the global existence of Theorem \ref{Theorem 1.1}, the following estimate holds:
\begin{equation}\label{4.29}
\|(\tilde{\rho},\mathbf{u})\|^l_{\dot{B}^{\frac{d}{2}-1}_{2,1}}+\|\tilde{\rho}\|^h_{\dot{B}^{\frac{d}{p}}_{p,1}}+\|\mathbf{u}\|^h_{\dot{B}^{\frac{d}{p}-1}_{p,1}}\leq c_0,
\end{equation}
from which we can infer from (\ref{4.28}) that there exists a constant $\bar c>0$ such that
\begin{equation}\label{4.30}
\frac{d}{dt}\Big(\|(\tilde{\rho},\mathbf{u})\|^l_{\dot{B}^{\frac{d}{2}-1}_{2,1}}+\|\tilde{\rho}\|^h_{\dot{B}^{\frac{d}{p}}_{p,1}}+\|\mathbf{u}\|^h_{\dot{B}^{\frac{d}{p}-1}_{p,1}}\Big)
+\bar{c}\Big(\|(\tilde{\rho},\mathbf{u})\|_{\dot B_{2,1}^{\frac{d}{2}+1}}^l+\|\tilde{\rho}\|_{\dot B_{p,1}^{\frac{d}{p}}}^h+\|\mathbf{u}\|_{\dot B_{p,1}^{\frac{d}{p}+1}}^h\Big)\leq 0.
\end{equation}
With (\ref{4.30}) in hand, one can get a Lyapunov-type differential inequality, which relies heavily on an interpolation inequality. Before using the interpolation inequality, it requires the uniform bound as follows:
\begin{equation}\label{4.31}
\|(\tilde{\rho},\mathbf{u})\|_{\dot B_{2,\infty}^{-\sigma}}^l\leq C,\ for\ any\ 1-\frac{d}{2}<\sigma\leq\sigma_0\triangleq\frac{2d}{p}-\frac{d}{2}.
\end{equation}

\subsection{Propagation the regularity of the initial data with negative index}
\quad\quad In this subsection, we will prove (\ref{4.31}). At first, we introduce a lemma.

\begin{lemma}\label{le3.1}
	Assume that $(\tilde{\rho},\mathbf{u})$ satisfies \eqref{1.4}, then we have
\begin{equation}	\big(\|(\tilde{\rho},\mathbf{u})\|_{\dot{B}^{-\sigma}_{2,\infty}}^l\big)^2\lesssim\big(\|(\tilde{\rho}_0,\mathbf{u}_0)\|_{\dot{B}^{-\sigma}_{2,\infty}}^l\big)^2+\int_{0}^{t}\|(N_1,N_2)\|_{\dot{B}^{-\sigma}_{2,\infty}}^l\|(\tilde{\rho},\mathbf{u})(\tau)\|_{\dot{B}^{-\sigma}_{2,\infty}}^ld\tau.
\end{equation}
\end{lemma}
\begin{proof}
We use the standard energy method. Consider the system
\begin{equation}\label{3.6}
	\begin{cases}
		\partial_t\tilde{\rho}+\alpha_2{\rm div}\mathbf{u}=N_1,\\
		\partial_t\mathbf{u}-\alpha_3\Delta \mathbf{u}-\alpha_4\nabla{\rm div}\mathbf{u}+\alpha_1\nabla\rho+\gamma\nabla\phi=N_2,\\
		-\Delta\phi+\phi=\tilde{\rho}.
	\end{cases}
\end{equation}
Applying $ \dot{\Delta}_j $ to $ \eqref{3.6}_1 $ and $\eqref{3.6}_2$, then taking the $L^2$ inner product with $\dot{\Delta}_j\tilde{\rho}$ and $\dot{\Delta}_j\mathbf{u}$ respectively, then we obtain
\begin{equation}\label{3.7}
		\int_{\mathbb{R}^d}\partial_t\dot{\Delta}_j\tilde{\rho}\dot{\Delta}_j\tilde{\rho}dx+\alpha_2\int_{\mathbb{R}^d}{\rm div}\dot{\Delta}_j\mathbf{u}\dot{\Delta}_j\tilde{\rho}dx=\int_{\mathbb{R}^d}\dot{\Delta}_jN_1\dot{\Delta}_j\tilde{\rho}dx,
\end{equation}
and
\begin{equation}\label{3.8}
	\begin{split}
		\int_{\mathbb{R}^d}\partial_t\dot{\Delta}_j\mathbf{u}\dot{\Delta}_j\mathbf{u}dx&-\alpha_3\int_{\mathbb{R}^d}\Delta\dot{\Delta}_j\mathbf{u}\dot{\Delta}_j\mathbf{u}dx-\alpha_4\int_{\mathbb{R}^d}\nabla{\rm div}\dot{\Delta}_j\mathbf{u}\dot{\Delta}_j\mathbf{u}dx\\
		&+\alpha_1\int_{\mathbb{R}^d}\nabla\dot{\Delta}_j\tilde{\rho}\dot{\Delta}_j\mathbf{u}dx+\gamma\int_{\mathbb{R}^d}\nabla\dot{\Delta}_j\phi\dot{\Delta}_j\mathbf{u}dx=\int_{\mathbb{R}^d}\dot{\Delta}_jN_2\dot{\Delta}_j\mathbf{u}dx.
	\end{split}
\end{equation}
Because of $ \eqref{3.6}_1 $ and $ \eqref{3.6}_3 $, the term $\gamma\int_{\mathbb{R}^d}\nabla\dot{\Delta}_j\phi\dot{\Delta}_j\mathbf{u}dx$ in the above equation can be converted to
\begin{equation}\label{3.9}
	\begin{split}
		\gamma\int_{\mathbb{R}^d}\nabla\dot{\Delta}_j\phi\dot{\Delta}_j\mathbf{u}dx=&-\gamma\int_{\mathbb{R}^d}\dot{\Delta}_j\phi{\rm div}\dot{\Delta}_j\mathbf{u}dx=\frac{\gamma}{\alpha_2}\int_{\mathbb{R}^d}\dot{\Delta}_j\phi_j\partial_t\dot{\Delta}_j\tilde{\rho}-\dot{\Delta}_j\phi\dot{\Delta}_jN_1dx\\
		=&\frac{\gamma}{\alpha_2}\int_{\mathbb{R}^d}\dot{\Delta}_j\phi_j\partial_t(-\Delta\dot{\Delta}_j\phi+\dot{\Delta}_j\phi)dx-\frac{\gamma}{\alpha_2}\int_{\mathbb{R}^d}\dot{\Delta}_j\phi\dot{\Delta}_jN_1dx\\
		=&\frac{\gamma}{\alpha_2}\frac{1}{2}\frac{d}{dt}\int_{\mathbb{R}^d}|\nabla\dot{\Delta}_j\phi|^2+(\dot{\Delta}_j\phi)^2dx-\frac{\gamma}{\alpha_2}\int_{\mathbb{R}^d}\dot{\Delta}_j\phi\dot{\Delta}_jN_1dx.
	\end{split}
\end{equation}
After plugging \eqref{3.9} into \eqref{3.8}, we multiply by $\frac{\alpha_1}{\alpha_2}$ on the \eqref{3.7} and sum up with \eqref{3.8} to get
\begin{equation}\label{3.10}
	\begin{split}	&\frac{1}{2}\frac{d}{dt}\big(\frac{\alpha_1}{\alpha_2}\|\dot{\Delta}_j\tilde{\rho}\|_{L^2}^2+\|\dot{\Delta}_j\mathbf{u}\|_{L^2}^2+\frac{\gamma}{\alpha_2}\|\nabla\dot{\Delta}_j\phi\|_{L^2}^2+\frac{\gamma}{\alpha_2}\|\dot{\Delta}_j\phi\|_{L^2}^2\big)\\
		&+\alpha_3\|\nabla \dot{\Delta}_j\mathbf{u}\|_{L^2}^2+\alpha_4\|{\rm div}\dot{\Delta}_j\mathbf{u}\|_{L^2}^2\\ \lesssim&\frac{|\alpha_1|}{\alpha_2}\|\dot{\Delta}_jN_1\|_{L^2}\|\dot{\Delta}_j\tilde{\rho}\|_{L^2}+\|\dot{\Delta}_jN_2\|_{L^2}\|\dot{\Delta}_j\mathbf{u}\|_{L^2}+\frac{|\gamma|}{\alpha_2}\|\dot{\Delta}_j\phi\|_{L^2}\|\dot{\Delta}_jN_1\|_{L^2}.
	\end{split}
\end{equation}
Besides, taking the Fourier transform to $ \eqref{3.6}_1 $ to have that
\begin{equation}
	|\xi|^2\hat{\phi}+\hat{\phi}=\hat{\tilde{\rho}}.
\end{equation}
In the low frequency, it holds that $|\xi|\le2^j(j\le j_0$). When $j_0$ is small enough, there exist two positive constants $C_1$ and $C_2$ such that
\begin{equation}
	C_1\|\hat{\tilde{\rho}}\|_{L^2}\le\|\hat{\phi}\|_{L^2}\le C_2\|\hat{\tilde{\rho}}\|_{L^2}.
\end{equation}
Owing to Plancherel theorem, it yields
\begin{equation}
	C_1\|\tilde{\rho}\|_{L^2}\le\|\phi\|_{L^2}\le C_2\|\tilde{\rho}\|_{L^2}.
\end{equation}
Hence, we can transfer \eqref{3.10} to the following inequality
\begin{equation}\label{3.14}
	\begin{split}		&\frac{1}{2}\frac{d}{dt}\big(\frac{\alpha_1+\gamma}{\alpha_2}\|\dot{\Delta}_j\tilde{\rho}\|_{L^2}^2+\|\dot{\Delta}_j\mathbf{u}\|_{L^2}^2\big)+\alpha_3\|\nabla \dot{\Delta}_j\mathbf{u}\|_{L^2}^2+\alpha_4\|{\rm div}\dot{\Delta}_j\mathbf{u}\|_{L^2}^2\\
		\lesssim&\frac{|\alpha_1|+|\gamma|}{\alpha_2}\|\dot{\Delta}_jN_1\|_{L^2}\|\dot{\Delta}_j\tilde{\rho}\|_{L^2}+\|\dot{\Delta}_jN_2\|_{L^2}\|\dot{\Delta}_j\mathbf{u}\|_{L^2}.
	\end{split}
\end{equation}
After that, multiplying by $2^{-2j\sigma}$ and taking $\sup_{j\leq j_0}$($j_0$ small enough) on both hand side of
\eqref{3.14}, we have
\begin{equation*}
	\begin{split}		&\frac{1}{2}\frac{d}{dt}\Big\{\frac{\alpha_1+\gamma}{\alpha_2}(\|\tilde{\rho}\|_{\dot{B}^{-\sigma}_{2,\infty}}^l)^2+(\|\mathbf{u}\|_{\dot{B}^{-\sigma}_{2,\infty}}^l)^2\Big\} \lesssim\frac{|\alpha_1|+|\gamma|}{\alpha_2}\|N_1\|^l_{\dot{B}^{-\sigma}_{2,\infty}}\|\tilde{\rho}\|^l_{\dot{B}^{-\sigma}_{2,\infty}}+\|N_2\|^l_{\dot{B}^{-\sigma}_{2,\infty}}\|\mathbf{u}\|^l_{\dot{B}^{-\sigma}_{2,\infty}}.
	\end{split}
\end{equation*}
When $\alpha_1+\gamma=\frac{P'(\bar{\rho})+\gamma\bar{\rho}}{\bar{\rho}}>0$, namely $P'(\bar{\rho})+\gamma\bar{\rho}>0$, we can arrive at
\begin{equation}
	\frac{d}{dt}\big(\|(\tilde{\rho},\mathbf{u})\|_{\dot{B}^{-\sigma}_{2,\infty}}^l\big)^2\lesssim\|(N_1,N_2)\|_{\dot{B}^{-\sigma}_{2,\infty}}^l\|(\tilde{\rho},\mathbf{u})\|_{\dot{B}^{-\sigma}_{2,\infty}}^l.
\end{equation}
Consequently, integrating the above inequality from 0 to $t$, it holds
\begin{equation} \big(\|(\tilde{\rho},\mathbf{u})\|_{\dot{B}^{-\sigma}_{2,\infty}}^l\big)^2\lesssim\big(\|(\tilde{\rho}_0,\mathbf{u}_0)\|_{\dot{B}^{-\sigma}_{2,\infty}}^l\big)^2+\int_{0}^{t}\|(N_1,N_2)\|_{\dot{B}^{-\sigma}_{2,\infty}}^l\|(\tilde{\rho},\mathbf{u})(\tau)\|_{\dot{B}^{-\sigma}_{2,\infty}}^ld\tau.
\end{equation}
This completes the proof.
\end{proof}

\begin{proposition}\label{proposition 3.1}
	Let $(\tilde{\rho},\mathbf{u})$ be the solutions constructed in Section $2$. For any $1-\frac{d}{2}<\sigma\leq\sigma_0$, then there exists a constant $c_1>0$ depends on the norm of the initial data such that for all $t\geq 0$,
	\begin{equation}\label{4.1}
		\|(\tilde{\rho},\mathbf{u})(t,\cdot)\|_{\dot{B}_{2,\infty}^{-\sigma}}^l\leq c_1.
	\end{equation}
\end{proposition}
\begin{proof}
	To begin with, we recall the following inequality from Lemma \ref{le3.1}:
	\begin{equation}
		\big(\|(\tilde{\rho},\mathbf{u})\|_{\dot{B}^{-\sigma}_{2,\infty}}^l\big)^2\lesssim\big(\|(\tilde{\rho}_0,\mathbf{u}_0)\|_{\dot{B}^{-\sigma}_{2,\infty}}^l\big)^2+\int_{0}^{t}\|(N_1,N_2)\|_{\dot{B}^{-\sigma}_{2,\infty}}^l\|(\tilde{\rho},\mathbf{u})(\tau)\|_{\dot{B}^{-\sigma}_{2,\infty}}^ld\tau.
	\end{equation}
Thus we focus on the nonlinear norm $\|(N_1,N_2)\|_{\dot{B}^{-\sigma}_{2,\infty}}^l$. For convenience, we divide $N_1$ and $N_2$ into low-frequency and high-frequency as follows:
\begin{equation*}
	N_1^l\triangleq-\mathbf{u}\cdot\nabla \tilde{\rho}^l-\tilde{\rho}{\rm div}\mathbf{u}^l,
\end{equation*}
\begin{equation*}
	N_2^l\triangleq-\mathbf{u}\cdot\nabla \mathbf{u}^l-F(\tilde{\rho})\nabla \tilde{\rho}^l-\alpha_3k(\tilde{\rho})\Delta \mathbf{u}^l-\alpha_4k(\tilde{\rho})\nabla{\rm div}\mathbf{u}^l,
\end{equation*}
and
\begin{equation*}
	N_1^h\triangleq-\mathbf{u}\cdot\nabla \tilde{\rho}^h-\tilde{\rho}{\rm div}\mathbf{u}^h,
\end{equation*}
\begin{equation*}
	N_2^h\triangleq-\mathbf{u}\cdot\nabla \mathbf{u}^h-F(\tilde{\rho})\nabla \tilde{\rho}^h-\alpha_3k(\tilde{\rho})\Delta \mathbf{u}^h-\alpha_4k(\tilde{\rho})\nabla{\rm div}\mathbf{u}^h.
\end{equation*}
We first handle $N_1^l$ and $N_2^l$. By using Corollary \ref{A.9} and the decomposition $\mathbf{u}=\mathbf{u}^l+\mathbf{u}^h$, we find that
\begin{equation*}
	\|\mathbf{u}^l\cdot\nabla\tilde{\rho}^l\|_{\dot{B}_{2,\infty}^{-\sigma}}\lesssim\|\nabla\tilde{\rho}^l\|_{\dot{B}^{\frac{d}{p}}_{p,1}}\|\mathbf{u}^l\|_{\dot{B}_{2,\infty}^{-\sigma}}\lesssim\|\tilde{\rho}\|^l_{\dot{B}^{\frac{d}{2}+1}_{2,1}}\|\mathbf{u}\|^l_{\dot{B}_{2,\infty}^{-\sigma}},
\end{equation*}
as well as
\begin{equation*}
	\|\mathbf{u}^h\cdot\nabla\tilde{\rho}^l\|_{\dot{B}_{2,\infty}^{-\sigma}}\lesssim\|\mathbf{u}^h\|_{\dot{B}^{\frac{d}{p}}_{p,1}}\|\nabla\tilde{\rho}^l\|_{\dot{B}_{2,\infty}^{-\sigma}}\lesssim\|\mathbf{u}\|^h_{\dot{B}^{\frac{d}{p}+1}_{p,1}}\|\tilde{\rho}\|^l_{\dot{B}_{2,\infty}^{-\sigma}}.
\end{equation*}
Combining with the above two inequalities, we get
\begin{equation}\label{3.20}
	\|\mathbf{u}\cdot\nabla\tilde{\rho}^l\|_{\dot{B}_{2,\infty}^{-\sigma}}\lesssim(\|\tilde{\rho}^l\|_{\dot{B}^{\frac{d}{2}+1}_{2,1}}+\|\mathbf{u}^h\|_{\dot{B}^{\frac{d}{p}+1}_{p,1}})\|(\tilde{\rho},\mathbf{u})\|^l_{\dot{B}_{2,\infty}^{-\sigma}}\lesssim\mathscr{E}_1(t)\|(\tilde{\rho},\mathbf{u})\|^l_{\dot{B}_{2,\infty}^{-\sigma}}.
\end{equation}
Besides, $\tilde{\rho}{\rm div}\mathbf{u}^l$ and $\mathbf{u}\cdot\nabla \mathbf{u}^l$ are bounded similarly as follows:
\begin{equation}
	\|\tilde{\rho}{\rm div}\mathbf{u}^l\|_{\dot{B}_{2,\infty}^{-\sigma}}\lesssim(\|\mathbf{u}\|^l_{\dot{B}^{\frac{d}{2}+1}_{2,1}}+\|\tilde{\rho}\|^h_{\dot{B}^{\frac{d}{p}}_{p,1}})\|(\tilde{\rho},\mathbf{u})\|^l_{\dot{B}_{2,\infty}^{-\sigma}}\lesssim\mathscr{E}_1(t)\|(\tilde{\rho},\mathbf{u})\|^l_{\dot{B}_{2,\infty}^{-\sigma}},
\end{equation}
and
\begin{equation}
	\|\mathbf{u}\cdot\nabla \mathbf{u}^l\|^l_{\dot{B}_{2,\infty}^{-\sigma}}\lesssim(\|\mathbf{u}\|^l_{\dot{B}^{\frac{d}{2}+1}_{2,1}}+\|\mathbf{u}\|^h_{\dot{B}^{\frac{d}{p}+1}_{p,1}})\|\mathbf{u}\|^l_{\dot{B}_{2,\infty}^{-\sigma}}\lesssim\mathscr{E}_1(t)\|\mathbf{u}\|^l_{\dot{B}_{2,\infty}^{-\sigma}}.
\end{equation}
To control the term $F(\tilde{\rho})\nabla \tilde{\rho}^l$, we first write
\begin{equation}\label{3.23}
	F(\tilde{\rho})=F'(0)\tilde{\rho}+\widetilde{F}(\tilde{\rho})\tilde{\rho}\ \ {\rm with}\ \ \widetilde{F}(0)=0.
\end{equation}
Similar to \eqref{3.20}, the term $\tilde{\rho}\nabla\tilde{\rho}^l$ is bounded as
\begin{equation}\label{3.24}
	\|\tilde{\rho}\nabla\tilde{\rho}^l\|_{\dot{B}_{2,\infty}^{-\sigma}}\lesssim(\|\tilde{\rho}^l\|_{\dot{B}^{\frac{d}{2}+1}_{2,1}}+\|\tilde{\rho}^h\|_{\dot{B}^{\frac{d}{p}}_{p,1}})\|\tilde{\rho}\|^l_{\dot{B}_{2,\infty}^{-\sigma}}\lesssim\mathscr{E}_1(t)\|\tilde{\rho}\|^l_{\dot{B}_{2,\infty}^{-\sigma}}.
\end{equation}
Turn to $\widetilde{F}(\tilde{\rho})\tilde{\rho}\nabla\tilde{\rho}^l$, by using Proposition \ref{A.13}, Corollary \ref{A.9} and \eqref{3.24}, it holds that
\begin{equation}
	\begin{split}
		\|\widetilde{F}(\tilde{\rho})\tilde{\rho}\nabla\tilde{\rho}^l\|^l_{\dot{B}_{2,\infty}^{-\sigma}}\lesssim&\|\widetilde{F}(\tilde{\rho})\|_{\dot{B}^{\frac{d}{p}}_{p,1}}\|\tilde{\rho}\nabla\tilde{\rho}^l\|_{\dot{B}_{2,\infty}^{-\sigma}}\\
		\lesssim&\|\tilde{\rho}\|_{\dot{B}^{\frac{d}{p}}_{p,1}}(\|\tilde{\rho}^l\|_{\dot{B}^{\frac{d}{2}+1}_{2,1}}+\|\tilde{\rho}^h\|_{\dot{B}^{\frac{d}{p}}_{p,1}})\|\tilde{\rho}\|^l_{\dot{B}_{2,\infty}^{-\sigma}}
		\lesssim\mathscr{E}_\infty(t)\mathscr{E}_1(t)\|\tilde{\rho}\|^l_{\dot{B}_{2,\infty}^{-\sigma}}.
	\end{split}
\end{equation}
For $k(\tilde{\rho})\Delta \mathbf{u}^l$, we write $k(\tilde{\rho})=k'(0)\tilde{\rho}+\widetilde{k}(\tilde{\rho})\tilde{\rho}$, we have the results in the same way as $F(\tilde{\rho})\nabla \tilde{\rho}^l$:
\begin{equation}
	\|\tilde{\rho}\Delta \mathbf{u}^l\|_{\dot{B}_{2,\infty}^{-\sigma}}\lesssim(\|\mathbf{u}\|^l_{\dot{B}^{\frac{d}{2}+1}_{2,1}}+\|\tilde{\rho}\|^h_{\dot{B}^{\frac{d}{p}}_{p,1}})\|(\tilde{\rho},\mathbf{u})\|^l_{\dot{B}_{2,\infty}^{-\sigma}}\lesssim\mathscr{E}_1(t)\|(\tilde{\rho},\mathbf{u})\|^l_{\dot{B}_{2,\infty}^{-\sigma}},
\end{equation}
and
\begin{equation}\label{3.26}
	\begin{split}
		\|\widetilde{k}(\tilde{\rho})\tilde{\rho}\Delta \mathbf{u}^l\|_{\dot{B}_{2,\infty}^{-\sigma}}\lesssim&\|\widetilde{k}(\tilde{\rho})\|_{\dot{B}^{\frac{d}{p}}_{p,1}}\|\tilde{\rho}\Delta \mathbf{u}^l\|_{\dot{B}_{2,\infty}^{-\sigma}}\\	\lesssim&\|\tilde{\rho}\|_{\dot{B}^{\frac{d}{p}}_{p,1}}\!(\|\mathbf{u}\|^l_{\dot{B}^{\frac{d}{2}+1}_{2,1}}\!+\!\|\tilde{\rho}\|^h_{\dot{B}^{\frac{d}{p}}_{p,1}})\|(\tilde{\rho},\mathbf{u})\|^l_{\dot{B}_{2,\infty}^{-\sigma}}
		\lesssim\mathscr{E}_\infty(t)\mathscr{E}_1(t)\|(\tilde{\rho},\mathbf{u})\|^l_{\dot{B}_{2,\infty}^{-\sigma}}.
	\end{split}
\end{equation}
Of course, the last term $k(\tilde{\rho})\nabla{\rm div}\mathbf{u}^l$ has the same bound as in (\ref{3.26}).

Next, we start to deal with the terms of $N_1^h$ and $N_2^h$. Using Proposition \ref{A.14}, we have
\begin{equation*}
	\begin{split}
		\|\mathbf{u}\cdot\nabla \mathbf{u}^h\|^l_{\dot{B}_{2,\infty}^{-\sigma}}\lesssim&(\|\mathbf{u}\|^l_{\dot{B}^{\frac{d}{2}-1}_{2,1}}+\|\mathbf{u}\|^h_{\dot{B}^{\frac{d}{p}}_{p,1}})\|\nabla \mathbf{u}\|^h_{\dot{B}^{\frac{d}{p}-1}_{p,1}} \lesssim\|\mathbf{u}\|^l_{\dot{B}^{\frac{d}{2}-1}_{2,1}}\|\mathbf{u}\|^h_{\dot{B}^{\frac{d}{p}}_{p,1}}+\|\mathbf{u}\|^h_{\dot{B}^{\frac{d}{p}}_{p,1}}\|\mathbf{u}\|^h_{\dot{B}^{\frac{d}{p}}_{p,1}}\\ \lesssim&\|\mathbf{u}\|^l_{\dot{B}^{\frac{d}{2}-1}_{2,1}}\|\mathbf{u}\|^h_{\dot{B}^{\frac{d}{p}+1}_{p,1}}+\|\mathbf{u}\|^h_{\dot{B}^{\frac{d}{p}-1}_{p,1}}\|\mathbf{u}\|^h_{\dot{B}^{\frac{d}{p}+1}_{p,1}}
		\lesssim\mathscr{E}_\infty(t)\mathscr{E}_1(t).
	\end{split}
\end{equation*}
Similarly, $\mathbf{u}\cdot\nabla \tilde{\rho}^h$ and $\tilde{\rho}{\rm div}\mathbf{u}^h$ can be bounded by $\mathscr{E}_\infty(t)\mathscr{E}_1(t)$ from above in light of Proposition \ref{A.14}. Now turn to $F(\tilde{\rho})\nabla \tilde{\rho}^h$, we use \eqref{3.23} again. The first term is bounded that
\begin{equation}\label{3.29}
	\begin{split}
		\|\tilde{\rho}\nabla \tilde{\rho}^h\|^l_{\dot{B}_{2,\infty}^{-\sigma}}\lesssim&(\|\tilde{\rho}\|^l_{\dot{B}^{\frac{d}{2}-1}_{2,1}}+\|\tilde{\rho}\|^h_{\dot{B}^{\frac{d}{p}}_{p,1}})\|\nabla \tilde{\rho}\|^h_{\dot{B}^{\frac{d}{p}-1}_{p,1}}\\		\lesssim&(\|\tilde{\rho}\|^l_{\dot{B}^{\frac{d}{2}-1}_{2,1}}+\|\tilde{\rho}\|^h_{\dot{B}^{\frac{d}{p}}_{p,1}})\|\tilde{\rho}\|^h_{\dot{B}^{\frac{d}{p}}_{p,1}}
		\lesssim\mathscr{E}_\infty(t)\mathscr{E}_1(t)
	\end{split}
\end{equation}
by using Proposition \ref{A.14}. For $\widetilde{F}(\tilde{\rho})\tilde{\rho}\nabla \tilde{\rho}^h$, we use Proposition \ref{A.7}, Proposition \ref{A.13} and Proposition \ref{A.14} to have
\begin{equation}
	\begin{split}
		\|\widetilde{F}(\tilde{\rho})\tilde{\rho}\nabla \tilde{\rho}^h\|_{\dot{B}_{2,\infty}^{-\sigma}}\lesssim&(\|\widetilde{F}(\tilde{\rho})\tilde{\rho}\|^l_{\dot{B}^{\frac{d}{2}-1}_{2,1}}+\|\widetilde{F}(\tilde{\rho})\tilde{\rho}\|^h_{\dot{B}^{\frac{d}{p}}_{p,1}})\|\nabla \tilde{\rho}\|^h_{\dot{B}^{\frac{d}{p}-1}_{p,1}}\\
		\lesssim&(\|\widetilde{F}(\tilde{\rho})\tilde{\rho}\|^l_{\dot{B}^{\frac{d}{2}-1}_{2,1}}+\|\tilde{\rho}\|_{\dot{B}^{\frac{d}{p}}_{p,1}}\|\tilde{\rho}\|_{\dot{B}^{\frac{d}{p}}_{p,1}})\|\tilde{\rho}\|^h_{\dot{B}^{\frac{d}{p}}_{p,1}}.
	\end{split}
\end{equation}
In order to bound  $\|\widetilde{F}(\tilde{\rho})\tilde{\rho}\|^l_{\dot{B}^{\frac{d}{2}-1}_{2,1}}$, we use Bony decomposition and \eqref{a.20} of Proposition \ref{A.13} to acquire
\begin{equation}
	\begin{split}
		\|T_{\widetilde{F}(\tilde{\rho})}\tilde{\rho}+R(\widetilde{F}(\tilde{\rho}),\tilde{\rho})\|^l_{\dot{B}^{\frac{d}{2}-1}_{2,1}}\lesssim&\|\widetilde{F}(\tilde{\rho})\|_{\dot{B}^{\frac{d}{p}-1}_{p,1}}\|\tilde{\rho}\|_{\dot{B}^{\frac{d}{p}}_{p,1}}\\
		\lesssim&(1+\|\tilde{\rho}\|_{\dot{B}^{\frac{d}{p}}_{p,1}})\|\tilde{\rho}\|_{\dot{B}^{\frac{d}{p}-1}_{p,1}}\|\tilde{\rho}\|_{\dot{B}^{\frac{d}{p}}_{p,1}}
		\lesssim(\mathscr{E}_\infty(t))^2,
	\end{split}
\end{equation}
and
\begin{equation}
	\|T_{\tilde{\rho}}\widetilde{F}(\tilde{\rho})\|^l_{\dot{B}^{\frac{d}{2}-1}_{2,1}}\lesssim\|\tilde{\rho}\|_{\dot{B}^{\frac{d}{p}-1}_{p,1}}\|\widetilde{F}(\tilde{\rho})\|_{\dot{B}^{\frac{d}{p}}_{p,1}}\lesssim\|\tilde{\rho}\|_{\dot{B}^{\frac{d}{p}-1}_{p,1}}\|\tilde{\rho}\|_{\dot{B}^{\frac{d}{p}}_{p,1}}\lesssim(\mathscr{E}_\infty(t))^2.
\end{equation}
Hence,
\begin{equation}\label{3.33}
	\begin{split}
		\|\widetilde{F}(\tilde{\rho})\tilde{\rho}\nabla \tilde{\rho}^h\|_{\dot{B}_{2,\infty}^{-\sigma}}\lesssim&(\|\widetilde{F}(\tilde{\rho})\tilde{\rho}\|^l_{\dot{B}^{\frac{d}{2}-1}_{2,1}}+\|\tilde{\rho}\|_{\dot{B}^{\frac{d}{p}}_{p,1}}\|\tilde{\rho}\|_{\dot{B}^{\frac{d}{p}}_{p,1}})\|\tilde{\rho}\|^h_{\dot{B}^{\frac{d}{p}}_{p,1}}\\
		\lesssim&(\mathscr{E}_\infty(t))^2\mathscr{E}_1(t)\lesssim\mathscr{E}_\infty(t)\mathscr{E}_1(t).
	\end{split}
\end{equation}
Then combining \eqref{3.29} and \eqref{3.33}, we have
\begin{equation}
	\|\mathbf{u}\cdot\nabla \tilde{\rho}^h\|^l_{\dot{B}_{2,\infty}^{-\sigma}}\lesssim\mathscr{E}_\infty(t)\mathscr{E}_1(t).
\end{equation}
Similarly, we estimate the term $k(\tilde{\rho})\Delta \mathbf{u}^h$ as follows:
\begin{equation}
	\begin{split}
		\|\tilde{\rho}\Delta \mathbf{u}^h\|_{\dot{B}_{2,\infty}^{-\sigma}}\lesssim&(\|\tilde{\rho}\|^l_{\dot{B}^{\frac{d}{2}-1}_{2,1}}+\|\tilde{\rho}\|^h_{\dot{B}^{\frac{d}{p}}_{p,1}})\|\Delta \mathbf{u}\|^h_{\dot{B}^{\frac{d}{p}-1}_{p,1}}\\
		\lesssim&(\|\tilde{\rho}\|^l_{\dot{B}^{\frac{d}{2}-1}_{2,1}}+\|\tilde{\rho}\|^h_{\dot{B}^{\frac{d}{p}}_{p,1}})\|\mathbf{u}\|^h_{\dot{B}^{\frac{d}{p}+1}_{p,1}}
		\lesssim\mathscr{E}_\infty(t)\mathscr{E}_1(t),
	\end{split}
\end{equation}
and
\begin{equation}
	\begin{split}
		\|\widetilde{k}(\tilde{\rho})\tilde{\rho}\Delta \mathbf{u}^h\|_{\dot{B}_{2,\infty}^{-\sigma}}\lesssim(\|\widetilde{k}(\tilde{\rho})\tilde{\rho}\|^l_{\dot{B}^{\frac{d}{2}-1}_{2,1}}+\|\widetilde{k}(\tilde{\rho})\tilde{\rho}\|^h_{\dot{B}^{\frac{d}{p}}_{p,1}})\|\nabla \tilde{\rho}\|^h_{\dot{B}^{\frac{d}{p}-1}_{p,1}}\lesssim\mathscr{E}_\infty(t)\mathscr{E}_1(t).
	\end{split}
\end{equation}
Therefore, it holds that
\begin{equation}
	\|k(\tilde{\rho})\Delta \mathbf{u}^h\|_{\dot{B}_{2,\infty}^{-\sigma}}\lesssim\mathscr{E}_\infty(t)\mathscr{E}_1(t).
\end{equation}
The last term $k(\tilde{\rho})\nabla{\rm div}\mathbf{u}^h$ is also bounded with $\mathscr{E}_\infty(t)\mathscr{E}_1(t)$ in the same way.
Now, we recall (3.19) and combine all the above inequalities to acquire
\begin{equation}
	\begin{split}
		\big(\|(\tilde{\rho},\mathbf{u})\|_{\dot{B}^{-\sigma}_{2,\infty}}^l\big)^2\lesssim&\big(\|(\tilde{\rho}_0,\mathbf{u}_0)\|_{\dot{B}^{-\sigma}_{2,\infty}}^l\big)^2 +\int_{0}^{t}(\mathscr{E}_1(t)+\mathscr{E}_\infty(t)\mathscr{E}_1(t))\big(\|(\tilde{\rho},\mathbf{u})(\tau)\|_{\dot{B}^{-\sigma}_{2,\infty}}^l\big)^2d\tau\\
		&+\int_{0}^{t}\mathscr{E}_\infty(t)\mathscr{E}_1(t)\|(\tilde{\rho},\mathbf{u})(\tau)\|_{\dot{B}^{-\sigma}_{2,\infty}}^ld\tau.
	\end{split}
\end{equation}
According to the definition of $\mathscr{E}_\infty(t)$ and $\mathscr{E}_1(t)$, it is obvious to see that
\begin{equation}
	\int_{0}^{t}\mathscr{E}_1(t)+\mathscr{E}_\infty(t)\mathscr{E}_1(t)d\tau\lesssim\int_{0}^{t}\mathscr{E}_1(t)d\tau+\sup_{t>0}\mathscr{E}_\infty(t)\int_{0}^{t}\mathscr{E}_1(t)d\tau\lesssim c_0.
\end{equation}
Eventually, through generalized Gronwall inequality in Proposition \ref{A.15}, we have the following
\begin{equation}
	\|(\tilde{\rho},\mathbf{u})(t,\cdot)\|_{\dot{B}_{2,\infty}^{-\sigma}}^l\leq c_1,\ {\rm for\ any}\ t>0,
\end{equation}
where $c_1>0$ depends on the norm of the initial data.
\end{proof}

\subsection{Lyapunov-type differential inequality}
\quad\quad In this subsection, we develop the Lyapunov-type inequality in time for energy norms, which leads to the time-decay estimates. On the one hand, owing to $-\sigma<\frac{d}{2}-1\le\frac{d}{p}<\frac{d}{2}+1$ and interpolation inequality, one has
\begin{equation}\label{4.32}
\|(\tilde{\rho},\mathbf{u})\|_{\dot B_{2,1}^{\frac{d}{2}-1}}^l\leq C\big(\|(\tilde{\rho},\mathbf{u})\|_{\dot B_{2,\infty}^{-\sigma}}\big)^{\eta_1}\big(\|(\tilde{\rho},\mathbf{u})\|_{\dot B_{2,1}^{\frac{d}{2}+1}}^l\big)^{1-\eta_1},
\end{equation}
where $\eta_1=\frac{4}{d+2\sigma+2}\in(0,1)$.
It together with Proposition \ref{proposition 3.1} results that
\begin{equation}\label{4.33}
\|(\tilde{\rho},\mathbf{u})\|_{\dot B_{2,1}^{\frac{d}{2}+1}}^l\geq c_1(\|(\tilde{\rho},\mathbf{u})\|_{\dot B_{2,1}^{\frac{d}{2}-1}}^l)^{\frac{1}{1-\eta_1}}.
\end{equation}
On the other hand, due to $\mathcal{X}(t)<2Cc_0<<1$ (where $c_0$ is small enough) and embedding relation in the high frequency, it is easy to see that
\begin{equation}\label{4.34}
\|\tilde{\rho}\|_{\dot B_{p,1}^{\frac{d}{p}}}^h\geq C\big(\|\tilde{\rho}\|_{\dot B_{p,1}^{\frac{d}{p}}}^h\big)^{\frac{1}{1-\eta_1}},\quad
\|\mathbf{u}\|_{\dot B_{p,1}^{\frac{d}{p}+1}}^h\geq C\big(\|\mathbf{u}\|_{\dot B_{p,1}^{\frac{d}{p}-1}}^h\big)^{\frac{1}{1-\eta_1}}.
\end{equation}
Thus, there exists a constant $\tilde{c}_0>0$ such that the following Lyapunov-type inequality holds
\begin{equation}\label{4.35}
	\begin{split}
		\frac{d}{dt}&\bigg(\|(\tilde{\rho},\mathbf{u})\|^l_{\dot B_{2,1}^{\frac{d}{2}-1}}+\|\tilde{\rho}\|_{\dot B_{p,1}^{\frac{d}{p}}}^h+\|\mathbf{u}\|_{\dot B_{p,1}^{\frac{d}{p}-1}}^h\bigg)\\
		&+\tilde{c}_0\bigg(\|(\tilde{\rho},\mathbf{u})\|^l_{\dot B_{2,1}^{\frac{d}{2}-1}}+\|\tilde{\rho}\|_{\dot B_{p,1}^{\frac{d}{p}}}^h+\|\mathbf{u}\|_{\dot B_{p,1}^{\frac{d}{p}-1}}^h\bigg)^{1+\frac{4}{d+2\sigma-2}}\leq 0.
	\end{split}
\end{equation}

\subsection{Decay estimate}
\quad\quad Solving the differential inequality (\ref{4.35}) directly, one gets
\begin{equation}\label{4.36}
\|(\tilde{\rho},\mathbf{u})\|^l_{\dot B_{2,1}^{\frac{d}{2}-1}}+\|\tilde{\rho}\|_{\dot B_{p,1}^{\frac{d}{p}}}^h+\|\mathbf{u}\|_{\dot B_{p,1}^{\frac{d}{p}-1}}^h\lesssim(1+t)^{-\frac{d+2\sigma-2}{4}}.
\end{equation}
For any $-\sigma-d(\frac{1}{2}-\frac{1}{p})<\sigma_1<\frac{d}{p}-1$, by the embedding theorem and interpolation inequality, we have
\begin{equation}\label{4.37}
\|(\tilde{\rho},\mathbf{u})\|_{\dot B_{p,1}^{\sigma_1}}^l\lesssim\|(\tilde{\rho},\mathbf{u})\|_{\dot B_{2,1}^{\sigma_1+d(\frac{1}{2}-\frac{1}{p})}}^l\lesssim(\|(\tilde{\rho},\mathbf{u})\|_{\dot B_{2,\infty}^{-\sigma}}^l)^{\eta_2}(\|(\tilde{\rho},\mathbf{u})\|_{\dot B_{2,1}^{\frac{d}{2}-1}}^l)^{1-\eta_2},
\end{equation}
where $\eta_2=\frac{\frac{d}{p}-1-\sigma_1}{\frac{d}{2}-1+\sigma}\in(0,1)$,  which combining with Proposition \ref{proposition 3.1} implies
\begin{equation}\label{4.38}
\|(\tilde{\rho},\mathbf{u})\|_{\dot B_{p,1}^{\sigma}}^l\lesssim(1+t)^{-\frac{(d+2\sigma_1-2)(1-\eta_2)}{4}}=(1+t)^{-\frac{d}{2}(\frac{1}{2}-\frac{1}{p})-\frac{\sigma_1+\sigma}{2}}.
\end{equation}
By virtue of $-\sigma-d(\frac{1}{2}-\frac{1}{p})<\sigma_1<\frac{d}{p}-1<\frac{d}{p}$, we find that
\begin{equation}\label{4.39}
\|(\tilde{\rho},\mathbf{u})\|^h_{\dot B_{p,1}^{\sigma_1}}\lesssim(\|\tilde{\rho}\|_{\dot B_{p,1}^{\frac{d}{p}}}^h+\|\mathbf{u}\|_{\dot B_{p,1}^{\frac{d}{p}-1}}^h)\lesssim(1+t)^{-\frac{d+2\sigma-2}{4}},
\end{equation}
and (\ref{4.38}) yields
\begin{equation}\label{3.40}
\begin{split}
	\|(\tilde{\rho},\mathbf{u})\|_{\dot B_{p,1}^{\sigma_1}}\lesssim&\|(\tilde{\rho},\mathbf{u})\|_{\dot B_{p,1}^{\sigma_1}}^l+\|(\tilde{\rho},\mathbf{u})\|_{\dot B_{p,1}^{\sigma_1}}^h\\
	\lesssim&(1+t)^{-\frac{d}{2}(\frac{1}{2}-\frac{1}{p})-\frac{\sigma_1+\sigma}{2}}+(1+t)^{-\frac{d+2\sigma-2}{4}}
	\lesssim(1+t)^{-\frac{d}{2}(\frac{1}{2}-\frac{1}{p})-\frac{\sigma_1+\sigma}{2}}.
\end{split}
\end{equation}
Hence, thanks to the embedding relation $\dot B_{p,1}^0\hookrightarrow L^p$, one also has
\begin{equation}\label{3.41}
\|\Lambda^{\sigma_1}(\tilde{\rho},\mathbf{u})\|_{L^p}\lesssim(1+t)^{-\frac{d}{2}(\frac{1}{2}-\frac{1}{p})-\frac{\sigma_1+\sigma}{2}}.
\end{equation}
This proves Theorem \ref{Theorem 1.2}.

\section{Instability}

\quad\quad In this section, we shall devote to study the instability of 3D linear and nonlinear problem for the system (\ref{1.1}) when $P'(\bar\rho)+\gamma\bar\rho<0$ based on the idea in \cite{guo2,jang,jiang3,jiang4,wangy}. The main process is similar to that in \cite{chen5} for a hyperbolic-parabolic model of vasculogenesis (has the damped mechanism). In particular, due to the nonlinear terms in \cite{chen5} are of the form $f\nabla g$, they can deduce the instability for the nonlinear problem in $H^3$-framework. In our case, if one studies the unknowns $(\rho,\mathbf{u})$ as usual, one can hardly deduce the instability in $H^3$-framework due to the presence of nonlinear term as $\rho\nabla^2\mathbf{u}$. Thus, sometimes we will study the conservative unknowns $(\rho,\mathbf{m})$ with the momentum $\mathbf{m}=(\rho+\bar\rho)\mathbf{u}$ although these two kinds of unknowns is equivalent for the small perturbation together with non-vacuum case. For simplicity, we denote $H^k:=H^k(\mathbb{R}^3)$ and $L^p:=L^p(\mathbb{R}^3)$ in the following.

\subsection{Linear instability}
\quad\quad  First, from (\ref{2.24}), one can immediately get the following:
After direct computations, we can get the following estimates for the eigenvalues.
\begin{lemma}\label{l 4.1}
When $P'(\bar\rho)+\gamma\bar\rho<0$ and $|\xi|\ll1$, we have
\begin{equation}\label{4.1}
\begin{array}{rl}
     \lambda_+(\xi)=\sqrt{-(P'(\bar\rho)+\gamma\bar\rho)}|\xi|-\frac{\eta}{2}|\xi|^2+\mathcal{O}(|\xi|^3),\\ \lambda_-(\xi)=-\sqrt{-(P'(\bar\rho)+\gamma\bar\rho)}|\xi|-\frac{\eta}{2}|\xi|^2-\mathcal{O}(|\xi|^3).
     \end{array}
     \end{equation}
When $|\xi|\gg1$, we have
     \begin{equation}\label{4.2}
    \lambda_+(\xi)=-\frac{P'(\bar\rho)}{\eta}+\mathcal{O}(|\xi|^{-2}),\ \ \
    \lambda_-(\xi)=-\eta|\xi|^2+\frac{P'(\bar\rho)}{\eta}+\mathcal{O}(|\xi|^{-2}).
    \end{equation}
When $0<\epsilon\leq|\xi|\leq R$ with any fixed constants $\epsilon$ and $R$, then there exists a $b>0$ such that
    \begin{equation}\label{4.3}
    Re(\lambda_\pm(\xi))\leq -b.
    \end{equation}
\end{lemma}

Lemma \ref{l 4.1} and the continuity of $\lambda_\pm(|\xi|)$ give that there exists a constant $\Theta>0$ and some $\xi_0\neq0$ such that
\begin{equation}\label{4.4}
\Theta=\max\{\sup\limits_{0<|\xi|<\infty}Re\lambda_+(|\xi|),\sup\limits_{0<|\xi|<\infty}Re\lambda_-(|\xi|)\}=\max\{Re\lambda_+(|\xi_0|),Re\lambda_-(|\xi_0|)\}.
\end{equation}
Moreover, if we suppose that $\lambda_0(|\xi|)$ is the eigenvalue whose real part is equal to $\Theta$ at $\xi_0$, then
\begin{equation}\label{4.5}
\Theta=Re\lambda_0(|\xi|)\geq Re\lambda_\pm(|\xi|),\ \ {\rm in\ some\ neighborhood\ of}\ \xi_0.
\end{equation}
Therefore, when $P'(\bar\rho)+\gamma\bar\rho<0$, one can immediately obtain for any $\eta_1$ and $\eta_2$ satisfying $0<|\eta_1|<|\eta_2|<\infty$ that
\begin{equation}\label{4.6}
Re\lambda_\pm\leq \Theta,\ \ \|e^{tA(|\xi|)}\hat{U}_0\|_{L^2(\eta_1\leq|\xi|\leq\eta_2)}\lesssim e^{\Theta t}\|U_0\|_{L^2},\ {\rm for\ any}\ t\geq0.
\end{equation}
Here $U:=(\rho-\bar\rho,\mathbf{u})$ or $U:=(\rho-\bar\rho,\mathbf{m})$ with the momentum $\mathbf{m}=\rho\mathbf{u}$.

With (\ref{4.5}) and (\ref{4.6}) in hand, we can derive the linear instability for some special initial data. In fact, from (\ref{4.5}), we know that for any $\bar\Theta\in(0,\Theta)$, there exists some positive constant $\bar\zeta=\bar\zeta(\bar\Theta)\leq\frac{|\xi_0|}{4}$ such that
\begin{equation}\label{4.7}
e^{t\lambda_0(\xi)}\geq e^{(\Theta-\bar\Theta)t},\ \ {\rm when}\ |\xi-\xi_0|\leq 2\bar\zeta.
\end{equation}
Now, we select the special initial data as
\begin{equation}\label{4.8}
\widehat{\rho_{0,\bar\Theta}^l}=\frac{\Psi(\xi)}{\|\Psi\|_{L^2}},\ \ \ \ \widehat{v_{0,\bar\Theta}^l}=-\frac{\lambda_0(|\xi|)\Psi(|\xi|)}{\bar\rho|\xi|\|\Psi\|_{L^2}},
\end{equation}
where $\Psi\in C_0^\infty(\mathbb{R}_\xi^3)$ and is a radial function satisfying $\Psi(\xi)=1$ for $|\xi-\xi_0|<\bar\zeta$ and $\Psi(\xi)=0$ for $|\xi-\xi_0|>2\bar\zeta$.

Then, we can get the linear instability as follows.

\begin{proposition}\label{l 4.2}
Assume that $P'(\bar\rho)+\gamma\bar\rho<0$. Let
\begin{equation}\label{4.9}
\rho_{\bar\Theta}^l=\mathcal{F}^{-1}[e^{\lambda_0(|\xi|)t}\widehat{\rho_{0,\bar\Theta}^l}],\ \ \ \mathbf{u}_{\bar\Theta}^l=\Lambda^{-1}\nabla\mathcal{F}^{-1}[e^{\lambda_0(|\xi|)t}\widehat{v_{0,\bar\Theta}}].
\end{equation}
Then $(\rho_{\bar\Theta}^l,\mathbf{u}_{\bar\Theta}^l)$ is a solution of the linear system (\ref{2.19}) with the initial data $(\rho_{0,\bar\Theta}^l,\mathbf{u}_{0,\bar\Theta}^l)$ depending on $\bar\Theta$. In addition, for any $\bar\Theta\in(0,\frac{\Theta}{2})$, one has
\begin{equation}\label{4.10}
\begin{split}
	e^{(\Theta-\bar\Theta)t}\|\rho_{0,\bar\Theta}^l\|_{L^2}\leq\|\rho_{\bar\Theta}^l\|\leq e^{\Theta t}\|\rho_{0,\bar\Theta}^l\|_{L^2},\\
e^{(\Theta-\bar\Theta)t}\|\mathbf{u}_{0,\bar\Theta}^l\|_{L^2}\leq\|\mathbf{u}_{\bar\Theta}^l\|\leq e^{\Theta t}\|\mathbf{u}_{0,\bar\Theta}^l\|_{L^2}.
\end{split}
\end{equation}
\end{proposition}
\begin{remark}Since the linear system for the unknowns $(\rho,\mathbf{m})$ is almost the same as $(\rho,\mathbf{u})$, one can get the same estimates in (\ref{4.1})-(\ref{4.10}) for $(\rho,\mathbf{m})$, which together with the conservation of $(\rho,\mathbf{m})$ can help us deduce the following nonlinear instability in $H^3$-framework.
\end{remark}

\subsection{Nonlinear instability}

\quad\quad We make a priori assumption that
\begin{equation}\label{4.11}
\mathcal{E}(t)=\|(\rho,\mathbf{u})\|_{H^3(\mathbb{R}^3)}\leq\delta\ll1,\ {\rm for\ any}\ t\geq0.
\end{equation}

In \cite{chen4}, under the assumptions $P'(\bar\rho)>0$, $\gamma>0$ and (\ref{4.11}), they get the following energy estimate:
\begin{equation}\label{4.12}
\mathcal{E}(t)\leq C\mathcal{E}(0),\ {\rm for\ any}\ t\geq0.
\end{equation}
We would like to emphasize that (\ref{4.12}) also holds under the more general stability condition $P'(\bar\rho)+\gamma\bar\rho>0$ in Section 3. In other words, we can also get the global existence of the solution for the system (\ref{1.1}) with the suitable small initial data in $H^3$-space.

We go back to the instability here. Without the stability condition $P'(\bar\rho)+\gamma\bar\rho>0$, one cannot get the desired estimate as in (\ref{4.12}). However, one can obtain a weaker estimate, which is important to deduce the nonlinear instability. In particular, one has

\begin{proposition}\label{l 4.3} Suppose that the hypotheses in Theorem \ref{l 1.4} holds. For any given initial data $(\rho_0-\bar\rho,\mathbf{u}_0)\in H^3(\mathbb{R}^3)$, there exists a constant $T>0$ and a unique solution $(\rho-\bar\rho,\mathbf{u})\in C^0(0,T;H^3(\mathbb{R}^3))$ to the Cauchy problem (\ref{1.4})-(\ref{1.5}). Additionally, there exist a constant $\bar\delta_0\in[0,1]$ and a constant $C$ independent of $T$, such that if $\mathcal{E}(t)\leq\bar\delta_0$ on $[0,T]$, then
\begin{equation}\label{4.13}
\mathcal{E}^2(t)\leq \mathcal{E}^2(0)+C\int_0^t\|(\rho-\bar\rho,\mathbf{u})\|_{L^2}^2d\tau,
\end{equation}
where $\mathcal{E}(t)$ is defined in (\ref{4.11}).
\end{proposition}
\begin{proof} The local existence is standard, and one can see that in \cite{chen4,chikami1}. The process of deriving the estimate (\ref{4.13}) is almost the same as the case under the stability condition $P'(\bar\rho)+\gamma\bar\rho>0$. The unique difference is that some terms containing the lowest order dissipation cannot be absorbed by the dissipation term in the left hand side of the energy inequality without the condition $P'(\bar\rho)+\gamma\bar\rho>0$. We omit the details here for simplicity.
\end{proof}

Up to now, we can begin to prove the nonlinear instability in Theorem \ref{l 1.4}.

\textit{\underline{Proof of Theorem \ref{l 1.4}}}. Denote $(\rho_{0,\bar\Theta}^\delta,\mathbf{u}_{0,\bar\Theta}^\delta):=\delta(\rho_{0,\bar\Theta}^l,\mathbf{u}_{0,\bar\Theta}^l)$ and $C_0:=\|(\rho_{0,\bar\Theta}^\delta,\mathbf{u}_{0,\bar\Theta}^\delta)\|_{H^3}$. From Proposition \ref{l 4.3}, there exists a $\tilde{\delta}\in(0,1)$ such that for any $\delta<\tilde{\delta}$, there exists a unique local solution $(\rho^\delta,\mathbf{u}^\delta)$ to the Cauchy problem (\ref{1.4})-(\ref{1.5}) with the initial data $(\rho_{0,\bar\Theta}^\delta,\mathbf{u}_{0,\bar\Theta}^\delta)$. In addition, let $\bar\delta_0$ be the same as in Proposition \ref{l 4.3} and $\delta_0=\min\{\tilde\delta,\bar\delta_0,\frac{1}{4C_0},,\frac{\bar\rho}{4C_0}\}$.

For any $\delta\in(0,\delta_0)$, take
\begin{equation}\label{4.15}
T^\delta=\frac{1}{\Theta}\ln\frac{2\epsilon_1}{\delta},\ \ \bar\Theta=\min\{\frac{\Theta}{2},\frac{1}{T^\delta}\},
\end{equation}
with $\epsilon_1=\min\Big\{\frac{\delta_0}{4C_1(C_0^2+\frac{1}{2\Theta\delta_0^2})^{1/2}},
\frac{1}{C_2^6\delta_0^4}\big(\min\{1,\frac{\bar\rho}{2}\}-2C_0\delta_0\big)^6,\frac{32\delta_0^2}{(C_2e)^6}\Big\}$ and $C_1,C_2$ to be determined later.	

Then, set
\begin{equation}\label{4.16}
T^*=\sup\{t\in(0,T^{max})\big|\sup\limits_{0\leq\tau\leq t}\mathcal{E}((\rho^\delta,\mathbf{u}^\delta)(\tau))\leq \delta_0\},
\end{equation}
and
\begin{equation}\label{4.17}
T^{**}=\sup\{t\in(0,T^{max})\big|\sup\limits_{0\leq\tau\leq t}\|(\rho^\delta,\mathbf{u}^\delta)(\tau)\|_{L^2}\leq \delta\delta_0^{-1}e^{\Theta t}\},
\end{equation}
where $T^{max}$ is the maximal time of existence. It is easy to see that $T^*T^{**}>0$, and
\begin{equation}
\mathcal{E}((\rho^\delta,\mathbf{u}^\delta)(T^*))=\delta_0\ \ {\rm if}\ T^*<\infty,
\end{equation}
and
\begin{equation}
\|(\rho^\delta,\mathbf{u}^\delta)(T^{**})\|_{L^2}=\delta\delta_0^{-1}e^{\Theta t}\ \ {\rm if}\ T^{**}<T^{max}.
\end{equation}

We verify that
\begin{equation}\label{4.19}
T^\delta\leq \min\{T^*,T^{**}\}.
\end{equation}
Suppose by way of contradiction that $T^*\leq T^\delta<\infty$. Thus by Proposition \ref{l 4.3}, there exists a positive constant $C_1$ such that
\begin{equation}
\begin{split}
\mathcal{E}((\rho^\delta,\mathbf{u}^\delta)(T^*))^2\leq \mathcal{E}((\rho^\delta,\mathbf{u}^\delta)(0))^2+C\int_0^T\|(\rho^\delta,\mathbf{u}^\delta)(t)\|_{L^2}^2dt\\
\leq C_0^2\delta^2+C\int_0^{T^*}(\delta\delta_0^{-1}e^{\Theta t})^2dt
\leq (C_0^2+\frac{Ce^{2\Theta T^*}}{2\Theta\delta_0^2})\delta^2\\
\leq C_1^2(C_0^2+\frac{1}{2\Theta \delta_0^2})e^{2\Theta T^\delta}\delta^2
= C_1^2(C_0^2+\frac{1}{2\Theta \delta_0^2})(2\epsilon_1)^2,
\end{split}
\end{equation}
which yields that
\begin{equation}
\mathcal{E}((\rho^\delta,\mathbf{u}^\delta)(T^*))\leq C_1\sqrt{C_0^2+\frac{1}{2\Theta \delta_0^2}}\times2\epsilon_1\leq \frac{\delta_0}{2},
\end{equation}
and hence this contradicts the definition of $T^*$.

To prove $T^\delta\leq T^{**}$ in $H^3$-framework, we shall use the conservative unknowns $(\rho,\mathbf{m})$ although it is equivalent to the unknowns $(\rho,\mathbf{u})$ when the perturbation is small and there is not vacuum. Therefore, when using the representation of the solution  $(\rho,\mathbf{m})$ with the cut-off initial data in the middle frequency defined in (\ref{4.8}), one can put one derivative of nonlinear term (conservative structure) on Green's function. In particular, suppose that $T^{**}\leq T^\delta$ by way of contradiction, one has from (\ref{4.6}) and (\ref{4.10}) that
\begin{equation}\label{4.22}
\begin{split}
&\|(\rho^\delta,\mathbf{m}^\delta)(T^{**})\|_{L^2}=\|(\widehat{\rho^\delta},\widehat{\mathbf{m}^\delta})(T^{**})\|_{L^2}
=\|(\widehat{\rho^\delta},\widehat{\tilde{v}^\delta})(T^{**})\|_{L^2}\\
\leq &\delta\|e^{\lambda_0(|\xi|)T^{**}}(\rho_{0,\bar\Theta}^l,\mathbf{u}_{0,\bar\Theta}^l)\|_{L^2}
+C\int_0^{T^{**}}\|e^{(T^{**}-t)A(|\xi|)}\mathcal{F}[\tilde{N}(\rho^\delta,\mathbf{m}^\delta)]\|_{L^2}dt\\
\leq &C_0\delta e^{\Theta T^{**}}+C\int_0^{T^{**}}e^{\Theta(T^{**}-t)}\|(\rho^\delta,\mathbf{m}^\delta)\|_{L^2}\|\nabla(\rho^\delta,\mathbf{m}^\delta)\|_{L^\infty}dt\\
\leq &C_0\delta e^{\Theta T^{**}}+C\int_0^{T^{**}}e^{\Theta(T^{**}-t)}
\|(\rho^\delta,\mathbf{m}^\delta)\|_{L^2}^{\frac{7}{6}}\|\nabla^3(\rho^\delta,\mathbf{m}^\delta)\|_{L^2}^{\frac{5}{6}}dt\\
\leq & C_0\delta e^{\Theta T^{**}}+C\int_0^{T^{**}}e^{\Theta(T^{**}-t)}\delta^{\frac{7}{6}}\delta_0^{-\frac{7}{6}}e^{\frac{7}{6}\Theta t}\delta_0^{\frac{5}{6}}dt\\
\leq &C_0\delta e^{\Theta T^{**}}+6Ce^{\Theta T^{**}}\delta^{\frac{7}{6}}\delta_0^{-\frac{1}{3}}\frac{e^{\frac{\Theta T^{**}}{6}}}{\Theta}\leq\delta\delta_0^{-1}e^{\Theta T^{**}}\Big[C_0\delta_0+\frac{C_2}{2}\delta_0^{\frac{2}{3}}\epsilon_1^{\frac{1}{6}}\big]\\
:=&\min\{\frac{1}{2},\frac{\bar\rho}{4}\}\delta\delta_0^{-1}e^{\Theta T^{**}}\Big[(C_0\delta_0+\frac{C_2}{2}\delta_0^{\frac{2}{3}}\epsilon_1^{\frac{1}{6}})\cdot(\min\{\frac{1}{2},\frac{\bar\rho}{4}\})^{-1}\Big]\\
\leq& \min\{\frac{1}{2},\frac{\bar\rho}{4}\}\delta\delta_0^{-1}e^{\Theta T^{**}},
\end{split}
\end{equation}
where we can choose suitable small $\delta_0$ such that $(C_0\delta_0+\frac{C_2}{2}\delta_0^{\frac{2}{3}}\epsilon_1^{\frac{1}{6}})\cdot(\min\{\frac{1}{2},\frac{\bar\rho}{4}\})^{-1}\leq 1$.

Then, (\ref{4.22}) together with the relation $\|\mathbf{u}^\delta\|_{L^2}\leq \frac{2}{\bar\rho}\|\mathbf{m}^\delta\|_{L^2}$ implies that
\begin{equation}
\|(\rho^\delta,\mathbf{u}^\delta)(T^{**})\|_{L^2}\leq \frac{1}{2}\delta\delta_0^{-1}e^{\Theta T^{**}}.
\end{equation}
This contradicts the definition of $T^{**}$. Therefore, we have proved (\ref{4.19}).

Finally, we shall complete the proof of Theorem \ref{l 1.4} by proving (\ref{1.16}). Similar to (\ref{4.22}), one has
\begin{equation}\label{4.26}
\begin{split}
&\|\rho^\delta(T^\delta)\|_{L^2}\geq \|\delta\widehat{\rho_{\bar\Theta}^l}\|_{L^2}
-C\int_0^{T^\delta}\|e^{(T^\delta-t)A(|\xi|)}\mathcal{F}[\tilde{N}(\rho^\delta,\mathbf{m}^\delta)]\|_{L^2}dt\\
\geq &\delta e^{(\Theta-\tilde{\Theta})T^\delta}\|\rho_{0,\bar\Theta}^l\|_{L^2}-\frac{C_2}{4}e^{\frac{7\Theta}{6}T^\delta}\delta^{\frac{7}{6}}\delta_0^{-\frac{1}{3}}\\
\geq &\delta e^{\Theta T^\delta}\big(e^{-\tilde{\Theta}T^\delta}-\frac{C_2}{4}e^{\frac{\Theta}{6}T^\delta}\delta^{\frac{1}{6}}\delta_0^{-\frac{1}{3}}\big)\\
\geq &\delta e^{\Theta T^\delta}\big(e^{-1}-\frac{C_2}{4}(2\epsilon_1)^{\frac{1}{6}}\delta_0^{-\frac{1}{3}}\big)\geq 2\epsilon_1\times\frac{1}{2}e^{-1}=\frac{\epsilon_1}{e},
\end{split}
\end{equation}
where we set $\tilde{\Theta}\leq\frac{1}{T^\delta}$, and use the relation $\epsilon_1\leq \frac{32\delta_0^2}{(C_2e)^6}$.

Additionally,
\begin{equation}\label{4.27}
\|\widehat{v_{0,\bar\Theta}^l}\|_{L^2}=\Big\|-\frac{\lambda_0(|\xi|)\Psi(\xi)}{\bar\rho|\xi|\|\Psi\|_{L^2}}\Big\|_{L^2}
\geq\inf\limits_{|\xi-\xi_0|<2\bar\zeta}\Big|-\frac{\lambda_0(|\xi|)}{\bar\rho|\xi|}\Big|\geq\frac{\Theta}{\bar\rho|\xi_0|},
\end{equation}
thus, a same way as in (\ref{4.26}) gives that
\begin{equation}\label{4.29}
\|\mathbf{u}^\delta(T^\delta)\|_{L^2}\geq \frac{\Theta}{\bar\rho|\xi_0|}\frac{\epsilon_1}{e}>0.
\end{equation}
This completes the proof of Theorem \ref{l 1.4} by taking $\epsilon_0=\min\{\frac{\epsilon_1}{e},\frac{\Theta}{\bar\rho|\xi_0|}\frac{\epsilon_1}{e}\}$.
			
\section {Appendix}

\quad\quad We will state several important lemmas and propositions on the homogeneous Besov space $\dot{B}^{s}_{p,1}$.
First, let $\mathcal {S}(\mathbb{R}^{d})$ be the Schwartz class of rapidly decreasing function. Given $f\in \mathcal
{S}(\mathbb{R}^{d})$, its Fourier transform $\mathcal {F}f=\widehat{f}$ is defined by $$\widehat{f}(\xi)=\int_{\mathbb{R}^{d}}e^{-ix\cdot\xi}f(x)dx.$$
Let $(\chi, \varphi)$ be a couple of smooth functions valued in $[0,1]$ such that $\chi$ is supported in the ball $\{\xi\in\mathbb{R}^{d}:  \ |\xi|\leq\frac{4}{3}\}$, $\varphi$ is supported in the shell $\{\xi\in \mathbb{R}^{d}: \ \frac{3}{4}\leq|\xi|\leq\frac{8}{3}\}$,  $\varphi(\xi):=\chi(\xi/2)-\chi(\xi)$
and
$$\chi(\xi)+\sum_{j\geq0}\varphi(2^{-j}\xi)=1~ \mathrm{for}\ \forall \ \xi \in\mathbb{ R}^{d},
~~~\sum_{j\in \mathbb{Z}}\varphi(2^{-j}\xi)=1 ~\mathrm{for} \ \forall \ \xi \in \mathbb{R}^{d}\setminus\{0\}.$$
For $f\in \mathcal{S}'$, the homogeneous frequency localization operators $\dot{\Delta}_j$ and $\dot{S}_j$ are defined by
\begin{equation*}
	\dot{\Delta}_{j}f\triangleq\varphi(2^{-j}D)f=\mathcal{F}^{-1}(\varphi(2^{-j}\xi)\mathcal{F}f)\quad {\rm and}\quad \dot{S}_{j}f\triangleq\chi(2^{-j}D)f=\mathcal{F}^{-1}(\chi(2^{-j}\xi)\mathcal{F}f).
\end{equation*}
We denote the space $\mathcal{S}'_{h}(\mathbb{R}^d)$ by the dual space of $\mathcal{S}'(\mathbb{R}^d)=\{f\in\mathcal{S}(\mathbb{R}^d):\,D^\alpha \hat{f}(0)=0\}$, which can also be identified by the quotient space of $\mathcal{S}'(\mathbb{R}^d)/{\mathbb{P}}$ with the polynomial space ${\mathbb{P}}$. The formal equality $$ f=\sum_{j\in\mathbb{Z}}\dot{\Delta}_jf $$ holds true for $f\in\mathcal{S}'_{h}(\mathbb{R}^d)$ and is called the homogeneous Littlewood-Paley decomposition, and then we have the fact that
\begin{equation}
	\dot{S}_jf=\sum_{q\le j-1}\dot{\Delta}_qf. \nonumber
\end{equation}
One easily verifies that with our choice of $\varphi$,
\begin{equation*}
	\dot{\Delta}_j\dot{\Delta}_qf\equiv0\quad \textrm{if}\quad|j-q|\ge
	2\quad \textrm{and} \quad
	\dot{\Delta}_j(\dot{S}_{q-1}f\dot{\Delta}_q f)\equiv0\quad \hbox{if}
	\quad |j-q|\ge 5.
\end{equation*}

\begin{definition}(Homogeneous Besov space)\label{A.1}
	For $s\in \mathbb{R}$ and $1\le p,r\le \infty$, the homogeneous Besov space $\dot{B}^{s}_{p,1}$ is defined by
	\begin{equation}\label{a.1}
		\dot{B}^s_{p,r}\triangleq\left\{f\in \mathcal{S}_h':\|f\|_{\dot{B}^s_{p,r}}<+\infty\right\},
	\end{equation}
	where
	\begin{equation}\label{a.2}
		\|f\|_{\dot{B}^s_{p,r}}\triangleq\|2^{js}\|\dot{\Delta}_jf\|_{L^p}\|_{\mathit{l}^r(\mathbb{Z})}.
	\end{equation}
\end{definition}

\begin{definition}(Chemin-Lerner spaces)\label{A.2}
	Let $T>0$, $s\in\mathbb{R}$, $1<r,p,q\le\infty$. The space $\widetilde{L}^q_{T}(\dot{B}^s_{p,r})$ is defined by
	\begin{equation}\label{a.3}
		\widetilde{L}^q_{T}(\dot{B}^s_{p,r})\triangleq\left\{f\in L^q(0,T;\mathcal{S}'_h):\|f\|_{\widetilde{L}^q_{T}(\dot{B}^s_{p,r})}<+\infty\right\},
	\end{equation}
	where
	\begin{equation}\label{a.4}
		\|f\|_{\widetilde{L}^q_{T}(\dot{B}^s_{p,r})}\triangleq\|2^{js}\|\dot{\Delta}_jf\|_{L^q(0,T;L^p)}\|_{\mathit{l}^r(\mathbb{Z})}.
	\end{equation}
\end{definition}
\begin{remark}\label{A.3}
	It holds that
	\begin{equation*}
		\|f\|_{\widetilde{L}^q_{T}(\dot{B}^s_{p,r})}\le\|f\|_{L^q_{T}(\dot{B}^s_{p,r})}\quad {\rm if}\quad r\ge q;\quad\|f\|_{\widetilde{L}^q_{T}(\dot{B}^s_{p,r})}\ge\|f\|_{L^q_{T}(\dot{B}^s_{p,r})}\quad {\rm if}\quad r\le q.
	\end{equation*}
\end{remark}
Restricting the above norms to the low or high frequencies parts of distributions will be crucial in our approach. For example, let us fix some integer $j_0$ and set
\begin{equation*}
	\|f\|^l_{\dot{B}^{s}_{p,1}}\triangleq\sum_{j\le j_0}2^{js}\|\dot{\Delta}_jf\|_{L^p},\quad \|f\|^h_{\dot{B}^{s}_{p,1}}\triangleq\sum_{j\ge j_0-1}2^{js}\|\dot{\Delta}_jf\|_{L^p};
\end{equation*}
\begin{equation*}
	\|f\|^l_{\widetilde{L}^\infty_{T}(\dot{B}^s_{p,1})}\triangleq\sum_{j\le j_0}2^{js}\|\dot{\Delta}_jf\|_{L^\infty_T(L^p)},\quad \|f\|^h_{\widetilde{L}^\infty_{T}(\dot{B}^s_{p,1})}\triangleq\sum_{j\ge j_0-1}2^{js}\|\dot{\Delta}_jf\|_{L^\infty_T(L^p)}.
\end{equation*}
\begin{lemma}(Bernstein inequalities)\label{A.4}
	Let $\mathscr{B}$ be a ball and $\mathscr{C}$ be a ring of $\mathbb{R}^d$. For $\lambda>0$, integer $k\ge0$, $1\le p\le q\le \infty$ and a smooth homogeneous function $\sigma$ in $\mathbb{R}^d\backslash\{0\}$ of degree $m$, then there holds
	\begin{equation*}
		\|\nabla^kf\|_{L^q}\le C^{k+1}\lambda^{k+d(\frac{1}{p}-\frac{1}{q})}\|f\|_{L^p},\quad {\rm whenever\ supp}\widehat{f}\subset\lambda\mathscr{B},
	\end{equation*}
	\begin{equation*}
		C^{-k-1}\lambda^k\|f\|_{L^q}\le\|\nabla^kf\|_{L^p}\le C^{k+1}\lambda^k\|f\|_{L^p},\quad {\rm whenever\ supp}\widehat{f}\subset\lambda\mathscr{C},
	\end{equation*}
	\begin{equation*}
		\|\sigma(\nabla)f\|_{L^q}\le C_{\sigma,m}\lambda^{m+d(\frac{1}{p}-\frac{1}{q})}\|f\|_{L^p},\quad {\rm whenever\ supp}\widehat{f}\subset\lambda\mathscr{C}.
	\end{equation*}
\end{lemma}

\begin{proposition}\cite{Bahouri-2011}(Embedding for Besov space on $\mathbb{R}^n$)\label{A.5}
	\begin{itemize}
		\item For any $p\in[1,\infty]$, we have the continuous embedding $\dot{B}^{0}_{p,1}\hookrightarrow L^p\hookrightarrow\dot{B}^{0}_{p,\infty}$.
		\item If $s\in\mathbb{R}$, $1\le p_1\le p_2\le \infty$, and $1\le r_1\le r_2\le \infty$ then  $\dot{B}^{s}_{p_1,r_1}\hookrightarrow\dot{B}^{s-d(\frac{1}{p_1}-\frac{1}{p_2})}_{p_2,r_2}$.
		\item The space $\dot{B}^{\frac{d}{p}}_{p,1}$ is continuously embedded in the set of bounded continuous function (going to zero at infinity if, additionally, $p<\infty$).
	\end{itemize}
\end{proposition}
\begin{proposition}\cite{danchin2,xu2023}\label{A.6}
	If supp$\mathcal{F}f\subset\left\{\xi\in\mathbb{R}^d:R_1\lambda\le|\xi|\le R_2\lambda\right\}$, then there exists $C$ depending only on $d$, $R_1$, $R_2$ so that for all $1<p<\infty$,
	\begin{equation}\label{a.5}
		C\lambda^2(\frac{p-1}{p})\int_{\mathbb{R}^d}|f|^pdx\le(p-1)\int_{\mathbb{R}^d}|\nabla f|^2|f|^{p-2}dx=-\int_{\mathbb{R}^d}\Delta f|f|^{p-2}fdx.
	\end{equation}
\end{proposition}
\begin{proposition}\cite{Bahouri-2011}(Interpolation inequality)\label{interpolation}
	Let $1\le p,r,r_1,r_2\le\infty$, if $f\in\dot{B}^{s_1}_{p,r_1}\cap\dot{B}^{s_2}_{p,r_2}$ and $s_1\neq s_2$, then $f\in\dot{B}^{\theta s_1+(1-\theta)s_2}_{p,r}$ for all $\theta\in(0,1)$ and
	\begin{equation}\label{a.6}
		\|f\|_{\dot{B}^{\theta s_1+(1-\theta)s_2}_{p,r}}\le\|f\|^\theta_{\dot{B}^{s_1}_{p,r_1}}\|f\|^{1-\theta}_{\dot{B}^{s_2}_{p,r_2}}
	\end{equation}
	with $\frac{1}{r}=\frac{\theta}{r_1}+\frac{1-\theta}{r_2}$.
\end{proposition}

\begin{proposition}\cite{Bahouri-2011,danchin5}\label{A.7}
	Let $s>0$, $1\leq p$, $r\leq \infty$, then $\dot{B}^{s}_{p,r}\cap L^\infty$ is an algerbra and
	\begin{equation}\label{a.7}
		\begin{array}{rl}
			\|fg\|_{\dot{B}^{s}_{p,r}} \lesssim \|f\|_{L^\infty}\|g\|_{\dot{B}^{s}_{p,r}} + \|g\|_{L^\infty}\|f\|_{\dot{B}^{s}_{p,r}}.
		\end{array}
	\end{equation}
	
	Let $s_1+s_2>0$, $s_1 \le \frac{d}{p_1}$, $s_2 \le \frac{d}{p_2}$, $s_1\ge s_2$, $\frac{1}{p_1}+\frac{1}{p_2}\le 1$. Then it holds that
	\begin{equation}\label{a.8}
		\|fg\|_{\dot{B}^{s_2}_{q,1}} \lesssim \|f\|_{\dot{B}^{s_1}_{p_1,1}}\|g\|_{\dot{B}^{s_2}_{p_2,1}},
	\end{equation}
	where $\frac{1}{q} = \frac{1}{p_1} + \frac{1}{p_2} - \frac{s_1}{d}$.
\end{proposition}
\begin{proposition}\cite{xin3}\label{A.8}
	Let the real numbers $s_1,\ s_2,\ p_1$ and $p_2$ be such that
	$$s_1+s_2\geq0,\ s_1\leq\frac{d}{p_1},\ s_2<\min\left(\frac{d}{p_1},\frac{d}{p_2}\right)\ and\ \frac{1}{p_1}+\frac{1}{p_2}\leq1.$$
	Then it holds that
	\begin{equation}\label{a.9}
		\|fg\|_{\dot B_{p_2,\infty}^{s_1+s_2-\frac{d}{p_1}}}\lesssim\|f\|_{\dot B_{p_1,1}^{s_1}}\|g\|_{\dot B_{p_2,\infty}^{s_2}}.
	\end{equation}
\end{proposition}

\begin{corollary}\label{A.9}
	Let the real numbers $1-\frac{d}{2}<\sigma_1\leq\sigma_0$ and $p$ satisfy (\ref{1.2.6}). The following two inequalities hold true:
	\begin{equation}\label{a.10}
		\|fg\|_{\dot B_{2,\infty}^{-\sigma_1}}\lesssim\|f\|_{\dot B_{p,1}^{\frac{d}{p}}}\|g\|_{\dot B_{2,\infty}^{-\sigma_1}},
	\end{equation}
	as well as
	\begin{equation}\label{a.11}
		\|fg\|_{\dot B_{2,\infty}^{\frac{d}{p}-\frac{d}{2}-\sigma_1}}\lesssim\|f\|_{\dot B_{p,1}^{\frac{d}{p}-1}}\|g\|_{\dot B_{2,\infty}^{\frac{d}{p}-\frac{d}{2}-\sigma_1+1}}.
	\end{equation}
\end{corollary}

\begin{proposition}\cite{xu2023}\label{A.10}
	Let $j_0\in\mathbb{Z}$, and denote $z^l\triangleq\dot S_{j_0}z,\ z^h\triangleq z-z^l$ and, for any $s\in\mathbb{R}$,
	$$\|z\|_{\dot B_{2,\infty}^{s}}^l\triangleq \sup_{j\leq j_0}2^{js}\|\dot\Delta_jz\|_{L^2}.$$
	There exists a universal integer $N_0$ such that for any $2\leq p\leq 4$ and $s>0$, we have
	\begin{equation}\label{a.12}
		\|fg^h\|_{\dot B_{2,\infty}^{-\sigma_0}}^l\leq C(\|f\|_{\dot B_{p,1}^{s}}+\|\dot S_{k_0+N_0}f\|_{L^{p^\ast}})\|g^h\|_{\dot B_{p,\infty}^{-s}},
	\end{equation}
	\begin{equation}\label{a.13}
		\|f^h g\|_{\dot B_{2,\infty}^{-\sigma_0}}^l\leq C(\|f^h\|_{\dot B_{p,1}^{s}}+\|\dot S_{k_0+N_0}f^h\|_{L^{p^\ast}})\|g\|_{\dot B_{p,\infty}^{-s}},
	\end{equation}
where $\sigma_0=\frac{2d}{p}-\frac{d}{2},\frac{1}{p*}+\frac{1}{p}=\frac{1}{2}$.
	
	Additionally, for exponents $s>0$, $1 \le p_1,p_2,q \le \infty$ and satisfy
	\begin{equation*}
		\frac{d}{p_1}+\frac{d}{p_2}-d\le s\le \min (\frac{d}{p_1},\frac{d}{p_2})\ \ and\ \ \frac{1}{q} = \frac{1}{p_1}+\frac{1}{p_2}-\frac{s}{d}.
	\end{equation*}
	Then it holds that
	\begin{equation}\label{a.14}
		\|fg\|_{\dot{B}^{-s}_{q,\infty}} \lesssim \|f\|_{\dot{B}^{s}_{p_1,1}}\|g\|_{\dot{B}^{-s}_{p_2, \infty}}.
	\end{equation}
\end{proposition}

\begin{proposition}\cite{xu2023}\label{A.11} The Bony decomposition satisfies that
	\begin{equation}\label{a.15}
		\|T_ab\|_{\dot{B}^{s-1+ \frac{d}{2} - \frac{d}{p}}_{2,1}} \lesssim \|a\|_{_{\dot{B}^{\frac{d}{p}-1}_{p,1}}}\|b\|_{\dot{B}^{s}_{p,1}},\ \ if\ d \ge 2\ and\ 1 \le p \le \min(4,\frac{2d}{d-2}),
	\end{equation}
	\begin{equation}\label{a.16}
		\|R(a,b)\|_{\dot{B}^{s-1+ \frac{d}{2} - \frac{d}{p}}_{2,1}} \lesssim \|a\|_{_{\dot{B}^{\frac{d}{p}-1}_{p,1}}}\|b\|_{\dot{B}^{s}_{p,1}},\ \ if\ s>1-\min(\frac{d}{p}+\frac{d}{p'})\ and\ 1\le p\le 4.
	\end{equation}
\end{proposition}

\begin{lemma}\cite{danchin5}\label{A.12}
	Let $d\geq2$, $1\le p, q\le \infty$, $v\in \dot{B}^{s}_{q,1}(\mathbb{R}^d)$ and $\nabla u \in \dot{B}^{\frac{d}{p}}_{p,1}(\mathbb{R}^d)$.\\
	Asumme that
	\begin{equation*}
		-d \min(\frac{1}{p}, 1-\frac{1}{q})<s\le 1+d\min(\frac{1}{p}, \frac{1}{q}).
	\end{equation*}
	Then it holds the commutator estimate
	\begin{equation}\label{a.17}
		\|[\dot{\Delta}_j, u\cdot\nabla]v\|_{L^p} \lesssim d_j2^{-js}\|\nabla u\|_{\dot{B}^{\frac{d}{p}}_{p,1}}\|v\|_{\dot{B}^{s}_{q,1}}.
	\end{equation}
	In the limit case $s=-d\min(\frac{1}{p}, 1-\frac{1}{q})$, we have
	\begin{equation}\label{a.18}
		\sup2^{js}\|[\dot{\Delta}_j, u\cdot\nabla]\|_{L^p}\lesssim \|\nabla u\|_{\dot{B}^{\frac{d}{p}}_{p,1}}\|v\|_{\dot{B}^{s}_{q,\infty}}.
	\end{equation}
\end{lemma}

\begin{proposition}\cite{Bahouri-2011}\label{A.13}
	Let $F: \mathbb{R} \mapsto \mathbb{R}$ be a smooth function with $F(0)=0$, $1\le p, r\le \infty$ and $s>0$. Then $F: \dot{B}^{s}_{p,r}(\mathbb{R}^d)\cap(L^\infty(\mathbb{R}^d))$ and
	\begin{equation}\label{a.19}
		\|F(u)\|_{\dot{B}^{s}_{p,r}}\le C\|u\|_{\dot{B}^{s}_{p,r}}
	\end{equation}
	with $C$ a constant depending only on $\||u\|_{L^\infty}$, $s$, $p$, $d$ and derivatives of $F$.
	
	If $s>-\min(\frac{d}{p}, \frac{d}{p'})$, then $F: \dot{B}^{s}_{p,r}(\mathbb{R}^d)\cap\dot{B}^{\frac{d}{p}}_{p,1}(\mathbb{R}^d) \mapsto \dot{B}^{s}_{p,r}(\mathbb{R}^d)\cap\dot{B}^{\frac{d}{p}}_{p,1}(\mathbb{R}^d)$, and
	\begin{equation}\label{a.20}
		\|F(u)\|_{\dot{B}^{s}_{p,r}}\le C(1+\|u\|_{\dot{B}^{\frac{d}{p}}_{p,1}})\|u\|_{\dot{B}^{s}_{p,r}}.
	\end{equation}
\end{proposition}

According to the analysis of Xu-Zhu \cite{xu2023}, there are three inequalities as follows:\\
If $2\le p < d$, then
\begin{equation*}
	\|fg^h\|^l_{\dot{B}^{-\sigma_1}_{2,\infty}}\lesssim\|fg^h\|^l_{\dot{B}^{-\sigma_0}_{2,\infty}}\lesssim\Big(\|f\|_{\dot{B}^{\frac{d}{p}-1}_{p,1}}+\|f^l\|_{L^{p^*}}\Big)\|g^h\|_{\dot{B}^{1-\frac{d}{p}}_{p,1}},
\end{equation*}
if $p=d$, then
\begin{equation*}
	\|fg^h\|^l_{\dot{B}^{-\sigma_1}_{2,\infty}}\lesssim\|fg^h\|^l_{\dot{B}^{-\sigma_0}_{2,\infty}}\lesssim\|f\|_{\dot{B}^{0}_{d,1}}\|g^h\|_{\dot{B}^{0}_{d,1}},
\end{equation*}
if $d<p\le2d$, then
\begin{equation*}
	\|fg^h\|^l_{\dot{B}^{-\sigma_1}_{2,\infty}}\lesssim\|fg^h\|^l_{\dot{B}^{-\sigma_0}_{2,\infty}}\lesssim\Big(\|f\|_{\dot{B}^{1-\frac{d}{p}}_{p,1}}+\|f^l\|_{L^{p^*}}\Big)\|g^h\|_{\dot{B}^{\frac{d}{p}-1}_{p,1}},
\end{equation*}
where $\frac{1}{p^*}\triangleq\frac{1}{2}-\frac{1}{p}$. In order to use the above inequalities expediently, we try to summarize these by using embedding $\dot{B}^{\frac{d}{p}}_{2,1}\hookrightarrow L^{p^*}$ to acquire a new inequality in a weaker version:
\begin{proposition}\label{A.14}
	If $2\le p\le2d$, then we have
	\begin{equation}
		\|fg^h\|^l_{\dot{B}^{-\sigma_1}_{2,\infty}}\lesssim\Big(\|f\|^l_{\dot{B}^{\frac{d}{2}-1}_{2,1}}+\|f\|^h_{\dot{B}^{\frac{d}{p}}_{p,1}}\Big)\|g\|^h_{\dot{B}^{\frac{d}{p}-1}_{p,1}},
	\end{equation}
where $\sigma_1$ satisfy $1-\frac{d}{2}<\sigma_1\le\sigma_0\triangleq\frac{2d}{p}-\frac{d}{2}$.
\end{proposition}

\begin{proposition}\cite{M}(Gronwall inequality)\label{A.15}
	Let $u(t)$ be a nonnegative function that satisfies the integral inequality
	\begin{equation}
		u(t)\le c+\int_{t_0}^{t}(a(s)u(s)+b(s)u^\alpha(s))ds,\ \ c\ge0,\ \ \alpha\ge0,
	\end{equation}
	where $a(t)$ and $b(t)$ are continuous nonnegative functions for $t\ge t_0$.
	For $0\le \alpha<1$ we have
	\begin{equation}
		\begin{split}
			u(t)\le & \{c^{1-\alpha}\exp \left[(1-\alpha\int_{t_0}^{t}a(s)ds)\right]\\
			&+(1-\alpha)\int_{t_0}^{t}b(s)\exp \left[(1-\alpha)\int_{s}^{t}a(r)dr\right]ds\}^{\frac{1}{1-\alpha}},
		\end{split}
	\end{equation}
	for $\alpha=1$,
	\begin{equation}
		u(t)\le c\exp\left\{\int_{t_0}^{t}[a(s)+b(s)]ds\right\}.
	\end{equation}
\end{proposition}

\section*{Acknowledgments}
		
\bigbreak
		
{\bf Funding}: The research was supported by National Natural Science Foundation of China (11971100) and Natural Science Foundation of Shanghai (22ZR1402300).\\
{\bf Conflict of Interest}: The authors declare that they have no conflict of interest.

\bibliographystyle{plain}

\begin{thebibliography}{99}

	
\bibitem{Bahouri-2011}H. Bahouri, J.Y.Chemin, R. Danchin, Fourier Analysis and Nonlinear Partial Differential Equations, Grundlehren der mathematischen, Grundlehren der mathematischen Wissenschaften, 343. Springer, Heidelberg, 2011.



\bibitem{charve}F. Charve, R. Danchin, A global existence result for the compressible Navier-Stokes
equations in the critical Lp framework, Arch. Ration. Mech. Anal. 198(2010), 233-271.

\bibitem{chen2}Q. Chen, C. Miao, Z. Zhang, Global well-posedness for compressible Navier-Stokes
equations with highly oscillating initial velocity, Comm. Pure Appl. Math. 63(2010), 1173-1224.

\bibitem{chen3}Q. Chen, C. Miao, Z. Zhang, On the ill-posedness of the compressible Navier-Stokes
equations in the critical Besov spaces, Rev. Mat. Iberoam. 31(2015), 1375-1402.

\bibitem{chen4}Q. Chen, G. Wu, Y. Zhang, L. Zou, Optimal time decay rates for the compressible Navier-stokes system with and without Yukawa-type potential, Electronic Journal of Differential Equations. 102(2020), 1-25.

\bibitem{chen5}Q. Chen, H.Q. Wang, G.C. Wu, On instability and stability of a quasi-linear hyperbolic-parabolic model for vasculogenesis, arXiv:2210.09497.


\bibitem{chikami1}N. Chikami, The blow-up criterion for the compressible Navier-Stokes system with a Yukawa-potential in the critical besov space, Differential and Integral Equations. 27(2014), 801-820.


\bibitem{danchin1}R. Danchin, Global existence in critical spaces for compressible Navier-Stokes equations, Invent. Math. 141(2000), 579-614.

%

\bibitem{danchin2}R. Danchin, Fourier analysis methods for the compressible Navier-Stokes equations, in Handbook of Mathematical Analysis in Mechanics of Viscous Fluids, Springer, Cham, Switzerland, 2016, 1-62.


\bibitem{danchin5}R. Danchin, J. Xu, Optimal time-decay estimates for the compressible Navier-Stokes equations in the critical $L^p$ framework, Arch. Ration. Mech Anal. 224(2017), 53-90.


\bibitem{duan1}R. Duan, H. Liu, S. Ukai, T. Yang, Optimal $L^p$-$L^q$ convergence rates for the compressible Navier-Stokes equations
with potential force, J. Diff. Eqns. 238(2007), 220-233.

%

\bibitem{ducomet} B. Ducomet, Simplified models of quantum fluids in nuclear physics, Proceedings of Partial
Differential Equations and Applications (Olomouc, 1999), Math. Bohem., 126(2)(2001), 323-336.

\bibitem{feireisl1}E. Feireisl, Compressible Navier-Stokes equations with a non-monotone pressure law, J. Diff. Eqns. 184(2002),  97-108.

%

\bibitem{feireisl4}E. Feireisl, A. Novotn\'y, H. Petzeltova\', On the global existence of globally defined weak solutions to the Navier-Stokes equations of isentropic compressible fluids, J. Math. Fluid Mech. 3(2001), 358-392.

\bibitem{guo1}Y. Guo, Y. J. Wang, Decay of dissipative equations and negative Sobolev spaces, Comm. Partial Differential Equations 37(2012), 2165-2208.

\bibitem{guo2}Y. Guo, W. Strauss, Instability of periodic BGK equilibria, Comm. Pure Appl. Math., 48(1995), 861-894.


\bibitem{haspot}B. Haspot, Existence of global strong solutions in critical spaces for barotropic viscous fluids, Arch. Ration. Mech. Anal. 202(2011), 427-460.

\bibitem{he}L. He, J. Huang, C. Wang, Global stability of large solutions to the 3D compressible
Navier-Stokes equations, Arch. Ration. Mech. Anal. 234(2019), 1167-1222.

\bibitem{hoff1}D. Hoff, H. Jenssen, Symmetric nonbarotropic flows with large data and forces, Arch. Ration. Mech. Anal. 173(2004), 297-343.





\bibitem{hoff2}D. Hoff, K. Zumbrun, Multi-dimensional diffusion waves for the Navier-Stokes equations of
compressible flow, Indiana U. Math. J. 44(1995), 603-676.

\bibitem{hoff3}D. Hoff, K. Zumbrun, Pointwise decay estimates for multidimensional Navier-Stokes diffusion waves, Z. Angew. Math. Phys. 48(1997), 597-614.

\bibitem{hu}X.P. Hu, G.C. Wu, Optimal rates of decay for solutions to the isentropic compressible
Navier-Stokes equations with discontinuous initial data, J. Diff. Eqns. 269 (2020),
8132-8172.

\bibitem{jang}J. Jang, I. Tice, Instability theory of the Navier-Stokes-Poisson equations, Analysis \& PDE, 6(2013), 1121-1181.

\bibitem{jiang3}F. Jiang, S. Jiang, On instability and stability of three-dimensional gravity driven viscous flows in a bounded
domain, Advances in Mathematics, 264 (2014), 831-863.

\bibitem{jiang4}F. Jiang, S. Jiang, Y. Wang, On the Rayleigh-Taylor instability for the incompressible viscous magnetohydrodynamic equations, Comm. Partial Differential Equations 39 (2014), 399-438.

\bibitem{jiang1}S. Jiang, Large-time behavior of solutions to the equations of a viscous polytropic ideal gas, Ann Math. Pure Appl. 175(1998), 253-275.

\bibitem{jiang2}S. Jiang, Large-time behavior of solutions to the equations of a one-dimensional viscous polytropic ideal gas in unbounded domains, Comm. Math. Phys. 200(1999), 181-193.



\bibitem{kawashima}S. Kawashima, System of a Hyperbolic-Parabolic Composite Type, with Applications to the
Equations of Magnetohydrodynamics, thesis, Kyoto University, Kyoto, 1983.

\bibitem{li}H.L. Li, T. Zhang, Large time behavior of isentropic compressible Navier-Stokes system in
R3, Math. Meth. Appl. Sci. 34 (2011), 670-682.

\bibitem{liu}T. P. Liu, W. K. Wang, The pointwise estimates of diffusion waves for the Navier-Stokes
equations in odd multi-dimensions, Comm. Math. Phys., 196 (1998), 145-173.

\bibitem{mastsumura1}A. Matsumura, T. Nishida, The initial value problem for the equations of motion of compressible viscous and heat-conductive fluids, Proc. Japan Acad. Ser. A 55 (1979), 337-342.

\bibitem{mastsumura2}A. Matsumura, T. Nishida, The initial value problem for the equations of motion of viscous
and heat-conductive gases, J Math Kyoto Univ 20(1980), 67-104.

\bibitem{M}D. S. Mitrinovic, J. E. Pecaric and A. M. Fink, Inequalities for Functions and Their Integrals and Derivatives, Mathematics and its Applications (East European Series), 53. Kluwer Academic Publishers Group, Dordrecht, 1991.

\bibitem{peng}H.Y. Peng, X.P. Zhai, The Cauchy problem for the $N$-dimensional compressible Navier-Stokes equations without heat conductivity, SIAM J. Math. Anal. 55(2)(2023), 1439-1463.

\bibitem{ponce}G. Ponce, Global existence of small solution to a class of nonlinear evolution equations,
Nonlinear Anal. 9(1985), 339-418.



\bibitem{wangy}Y. Wang, I. Tice, The viscous surface-internal wave problem: nonlinear Rayleigh-Taylor instability, Comm.
Partial Differential Equations, 37(2012), 1967-2028.

\bibitem{wen1}H. Wen, C. Zhu, Blow-up criterions of strong solutions to 3D compressible Navier-Stokes equations with vacuum, Adv. Math. 248(2013), 534-572.


\bibitem{wen3}H. Wen, C. Zhu, Global solutions to the three-dimensional full compressible Navier-Stokes equations with vacuum at infinity in some classes of large data, SIAM J. Math. Anal. 49(2017), 162-221.


\bibitem{xin1}Z. Xin, Blow up of smooth solutions to the compressible Navier-Stokes equation with compact density, Comm. Pure Appl. Math. 51(1998), 229-240.

\bibitem{xin2}Z. Xin, W. Yan, On blow-up of classical solutions to the compressible Navier-Stokes equations, Comm. Math. Phys. 321(2013), 529-541.

\bibitem{xin3}Z. Xin, J. Xu, Optimal decay for the compressible Navier-Stokes equations without additional smallness assumptions, J. Diff. Eqns. 274(2021), 543-575.

\bibitem{xu2023}J. Xu, L.M. Zhu, Global existence and optimal time decay for the viscous liquid-gas two-phase flow model in the $L^p$ critical Besov space, Discrete and Continuous Dynamical Systems-B 28(9)(2023), 5055-5086.


\bibitem{zhai}X. Zhai, Y. Li, F. Zhou, Global large solutions to the three dimensional compressible
Navier--Stokes equations, SIAM J. Math. Anal. 52(2020), 1806-1843.

	

			

			

			
			
			
		\end{thebibliography}

	\end{document}